\documentclass[10pt,reqno]{amsart}
\usepackage[top=30truemm,bottom=30truemm,left=25truemm,right=25truemm]{geometry}

\usepackage{amssymb, amsmath, amsthm, amsfonts, latexsym, mathtools}
\usepackage{graphicx, arydshln, enumerate}
\usepackage{xcolor}
\usepackage[all]{xy}
\usepackage{amscd, dsfont}
\usepackage{comment}
\usepackage{appendix}
\usepackage{hyperref}
\usepackage[nameinlink]{cleveref}
\hypersetup{
setpagesize=false,
bookmarksnumbered=true,%
bookmarksopen=true,%
colorlinks=true,%
linkcolor=orange,
citecolor=blue,
}
\usepackage{autonum}

\crefname{equation}{}{}

\newtheorem{thm}{Theorem}[section]
\crefname{thm}{Theorem}{Theorems}
\newtheorem{prop}[thm]{Proposition}
\crefname{prop}{Proposition}{Propositions}
\newtheorem{lem}[thm]{Lemma}
\crefname{lem}{Lemma}{Lemmas}
\newtheorem{cor}[thm]{Corollary}
\crefname{cor}{Corollary}{Corollaries}
\newtheorem{conj}[thm]{Conjecture}
\crefname{conj}{Conjecture}{Conjectures}
\newtheorem{prob}[thm]{Problem}

\theoremstyle{definition}
\newtheorem*{definition}{Definition}
\newtheorem*{example}{Example}

\theoremstyle{remark}

\newcommand{\til}{\widetilde}
\newcommand{\wh}{\widehat}

\newcommand{\ep}{\varepsilon}

\newcommand{\Tr}{\operatorname{Tr}}

\newcommand{\ceq}{\coloneqq}
\newcommand{\Z}{\mathbb{Z}}

\newcommand{\SL}{\operatorname{SL}}
\newcommand{\MES}{\mathcal{E}}
\newcommand{\MPF}{\mathcal{P}}

\newcommand{\Span}{\operatorname{Span}}

\newcommand{\C}{\mathbb{C}}
\newcommand{\R}{\mathbb{R}}
\newcommand{\Q}{\mathbb{Q}}
\newcommand{\HH}{\mathbb{H}}

\newcommand{\summ}[1]{\sum_{\substack{#1}}}

\numberwithin{equation}{section}

\newcommand\quotient[2]{
	\mathchoice
	{
		\text{\raise1ex\hbox{$#1$}\Big/\lower1ex\hbox{$#2$}}%
	}
	{
		#1\,/\,#2
	}
	{
		#1\,/\,#2
	}
	{
		#1\,/\,#2
	}
}

\begin{document}

\title{Multiple $\wp$-functions and their applications} 

\author[Kanno]{Hayato Kanno} 
\author[Kina]{Katsumi Kina}

\begin{abstract}
In this paper, we introduce and study multiple $\wp$-functions, which generalize the classical Weierstrass $\wp$-function to iterated sums over lattice points, and we establish explicit formulas expressing them in terms of single $\wp$-functions with coefficients given by multiple Eisenstein series. As an application, we derive some relations among multiple Eisenstein series and multiple zeta values by exploiting the double periodicity of the multiple $\wp$-functions.
\end{abstract}

\maketitle

\section{Introduction}

The purpose of this paper is to introduce and study the \emph{multiple $\wp$-function}, which can be viewed as a nested generalization of Weierstrass's $\wp$-function. This is defined for a non-negative integer $r\ge0$, positive integers $k_1,\dots,k_r\ge2$ and $\tau\in\HH$ as an iterated multiple sum by
\begin{equation}\label{eq:wp-def}
\wp_{k_1,\dots,k_r}(z;\tau)\ceq
\lim_{M\to\infty}\lim_{N\to\infty}\sum_{\substack{w_s\in\Z_M\tau+\Z_N\\ w_1\prec\dots\prec w_r}}\frac{1}{(z-w_1)^{k_1}(z-w_2)^{k_2}\cdots(z-w_r)^{k_r}},\qquad(r\ge1),
\end{equation}
and $\wp_\emptyset(z;\tau)\ceq1$ when $r=0$. Here, $\HH$ denotes the upper half-plane, $\Z_N=\{n\in\Z\mid |n|<N\}$, and the order $\prec$ on the lattice points in $L_\tau\ceq\Z\tau+\Z$ is defined by
\begin{align}
m_1\tau+n_1\prec m_2\tau+n_2 &\overset{\mathrm{def}}{\Longleftrightarrow}
\left\{\begin{array}{c}
m_1<m_2
\\
\text{or}
\\
m_1=m_2 \quad \text{and}\quad n_1<n_2
\end{array}\right. .
\end{align}
We call a tuple of positive integers $(k_1,\dots,k_r)$ the \emph{index}, $k_1+\cdots+k_r$ its \emph{weight} and $r$ its \emph{depth}. When the context is clear, we often omit $\tau$ and write $\wp_{k_1,\dots,k_r}(z)$ instead of $\wp_{k_1,\dots,k_r}(z;\tau)$.

For any integers $k_1,\dots,k_r\ge2$, the series \cref{eq:wp-def} converges on $\C\setminus L_\tau$, and defines a meromorphic function on $\C$ that has poles only at the lattice points in $L_\tau$. We show in \Cref{lem:doubly-multi-p} that our multiple $\wp$-function is an elliptic function whose poles occur only at $L_\tau$ and have an order at most $\max\{k_1,\dots,k_r\}$.
In this paper, we establish an explicit formula for multiple $\wp$-functions using in terms of the classical $\wp$-function, and derive several applications.
An important perspective is that multiple $\wp$-functions are closely related to \emph{multiple Eisenstein series}.
The multiple Eisenstein series $\til{G}_{k_1,\dots,k_r}(\tau)$, introduced by Gangl, Kaneko and Zagier (\cite{GKZ}) in the case $r=2$, is a holomorphic fention on the upper half plane, defined for $k_1,\dots,k_r\ge2$ by
\begin{align}
    \til{G}_{k_1,\dots,k_r}(\tau)\ceq\lim_{M\to\infty}\lim_{N\to\infty}\sum_{\substack{w_s\in\Z_M\tau+\Z_N\\ 0\prec w_1\prec\dots\prec w_r}}\frac{1}{w_1^{k_1}\cdots w_r^{k_r}}.
\end{align}
The study of multiple Eisenstein series is motivated in part by the modular phenomena of \emph{multiple zeta values} defined by
\begin{align}
\zeta(k_1,\dots,k_r)\ceq \sum_{0<n_1<\dots<n_r}\frac{1}{n_1^{k_1}\cdots n_r^{k_r}},
\end{align}
for positive integers $k_1,\dots,k_r$ with $k_r\geq 2$. Gangl--Kaneko--Zagier \cite{GKZ} and Kaneko \cite{Kaneko04} investigated the double shuffle relation satisfied by double zeta values for the double Eisenstein series $\til{G}_{k_1,k_2}(\tau)$ and described the \emph{modular phenomena} for double zeta values via double Eisenstein series. Their results were extended to the double Eisenstein series for higher level (congruence subgroup of level $N$) in \cite{KT} ($N=2$), \cite{Kina} ($N=4$) and in \cite{YZ} ($N$ : general), and have interesting applications to the theory of modular forms (see \cite{Tasaka}).
For higher depth, Bachmann--Tasaka \cite{BT} discovered a nontrivial correspondence between the Fourier expansion of multiple Eisenstein series and the Goncharov coproduct, which arises in a certain Hopf algebra, and also constructed shuffle regularized multiple Eisenstein series. Their results were extended for general level in \cite{Kanno}. In this paper, multiple Eisenstein series appear in the Taylor coefficients of multiple $\wp$-functions. In \Cref{sec:general-index}, we present some relations among multiple Eisenstein series via multiple $\wp$-functions.
The multiple $\wp$-functions can also be viewed as doubly periodic analogues of \emph{multitangent functions}\footnote{In \cite{Bouillot}, the multitangent functions are denoted by $\mathcal{T}e^{k_1,\dots,k_r}(z)$.} introduced by Bouillot \cite{Bouillot}, which are defined by
\begin{align}
    \Psi_{k_1,\dots,k_r}(z)\ceq\sum_{n_1<\cdots<n_r}\frac{1}{(z+n_1)^{k_1}\cdots(z+n_r)^{k_r}}
\end{align}
for $k_1,\dots,k_r\ge1$ with $k_1,k_r\ge2$.
He discovered that any multitangent function can be explicitly written as a linear combination of monotangent functions $\Psi_k(z)$ whose coefficients are multiple zeta values.
In the doubly periodic situation, we know that the analogous reduction occurs since the multiple $\wp$-functions are a elliptic functions. For instance, we have
\begin{align}
    \wp_{2,2}&=G_2\wp+\frac{1}{2}(G_2^2+5G_4),\\
    \wp_{2,2,2}&=\frac{1}{2}(G_2^2-G_4)\wp+\frac{1}{6}G_2^3+\frac{5}{2}G_2G_4-\frac{14}{3}G_6,\\
    \wp_{3,3}&=-3G_4\wp-\frac{21}{2}G_6,\\
    \wp_{2,3}&=-\frac{1}{2}\til{G}_2\wp^\prime-3\til{G}_3\wp-11\til{G}_5-2\til{G}_{2,3}
    \end{align}
where $\wp$ is the Weierstrass $\wp$-function, and $G_k$ is the Eisenstein series of weight $k$ defined respectively by
\begin{gather}
    \wp(z;\tau)\ceq\frac{1}{z^2}+\sum_{w\in L_\tau\setminus\{0\}}\left(\frac{1}{(z-w)^2}-\frac{1}{w^2}\right),\\
    G_k(\tau)\ceq\lim_{M\to\infty}\lim_{N\to\infty}\sum_{\substack{(m,n)\in\Z^2\setminus\{(0,0)\},\\|m|<M,|n|<N}}\frac{1}{(m\tau+n)^k}
\end{gather}
for $k\ge 2$. Note that $G_k$ vanishes when $k$ is odd, and $G_k=2\til{G}_k$ when $k$ is even. We obtained an explicit expression of the reduction with its coefficients as the multiple Eisenstein series, and this is the main result of this paper.

\begin{thm}[\cref{thm:depth-2-wp-formula}]\label{thm:rough-state-coeff-two}
For integers $k_1,\dots,k_r\geq 2$ with $k_1+\cdots+k_r=k$, we have
\begin{align}
\wp_{k_1,\dots,k_r}(z;\tau)
=&\sum_{i=1}^{r}\sum_{\substack{n_1,\dots,n_r\geq 2\\n_1+\cdots+n_r=k}}(-1)^{k_i+n_{i}+\cdots+n_r}\prod_{\substack{j=1\\j\neq i}}^{r}\binom{n_j-1}{k_j-1}
\til{G}_{n_{i-1},\dots,n_1}(\tau)\til{G}_{n_{i+1},\dots,n_r}(\tau)(\wp_{n_i}(z;\tau)-G_{n_i}(\tau))
\\
&+\sum_{i=1}^{r}\sum_{\substack{n_1,\dots,n_r\geq 0\\n_1+\cdots+n_r=k\\n_i=0}}(-1)^{k_i+n_{i+1}+\cdots+n_r}\prod_{\substack{j=1\\j\neq i}}^{r}\binom{n_j-1}{k_j-1}
\til{G}_{n_{i-1},\dots,n_1}(\tau)\til{G}_{n_{i+1},\dots,n_r}(\tau)
\\
&+\sum_{i=0}^{r}(-1)^{k_{i+1}+\cdots+k_r}\til{G}_{k_{i},\dots,k_1}(\tau)\til{G}_{k_{i+1},\dots,k_r}(\tau).
\end{align}
\end{thm}

In some special cases, quasi-modular forms appear in the coefficients of the reduction and they can be expressed via \emph{partition Eisenstein traces}, which have recently been studied. For example, in \cite{AOS}, Amdeberhan, Ono and Singh obtained expressions for two families of quasi-modular forms considered by Ramanujan in his Lost Notebook in terms of partition Eisenstein traces. Later, Matsusaka \cite{Mat} reinterpreted their results and obtained alternative expressions using partition Eisenstein traces. Related results can also be found in \cite{AGOS} and \cite{Bringmann}. 


In \Cref{sec:algebra}, we will also discuss the $\Q$-linear space spanned by multiple $\wp$-functions. In the case of single period, Bouillot \cite[Conjecture 3]{Bouillot} conjectured that the $\Q$-algebra of multitangent functions forms a module over the $\Q$-algebra of multiple zeta values, and Hirose \cite{Hirose} solved this conjecture. Bouillot also conjectured that this module is a graded module, and he noted it follows from so-called direct sum conjecture for multiple zeta values (\cite[Conjecture 10, Property 6]{Bouillot}). In this paper, we present analogous questions for the space of multiple $\wp$-functions.

Let $\MPF_k$ be the $\Q$-linear space spanned by multiple $\wp$-functions of weight $k$, and $\MPF$ be the space spanned by all multiple $\wp$-functions. The iterated sum expression produces an algebraic structure on $\MPF$. The product associated to this is called the harmonic (stuffle) product. For instance, we have
\begin{align}
    \wp_{2}\wp_{3}=\wp_{2,3}+\wp_{3,2}+\wp_{5}.
\end{align}
In general, we know that $\MPF_k\cdot\MPF_l\subset\MPF_{k+l}$ for any $k,l\ge2$, and therefore, $\MPF$ is a $\Q$-algebra. \Cref{thm:rough-state-coeff-two} and the fact of elliptic functions say that there is a canonical embedding
\begin{align}
    \MPF\hookrightarrow\left(\wp^\prime\Q[\wp]\oplus\Q[\wp]\right)\otimes\mathcal{E},
\end{align}
where $\mathcal{E}$ is the $\Q$-linear space spanned by all multiple Eisenstein series. Note that this morphism is not an isomorphism.

The organization of this paper is as follows. In \Cref{sec:ell func}, we review some facts from the classical theory of elliptic functions, and also explain some basic analytic properties of multiple $\wp$-functions. In \Cref{sec:explicit-formula}, we give some explicit expression of the multiple $\wp$-function with repeated index. A remarkable observation is that its coefficients are quasi-modular forms and can furthermore be expressed in terms of partition Eisenstein traces. We also give an explicit expression of the higher derivatives of classical $\wp$-function, which is due to a comment by Professor Toshiki Matsusaka. In \Cref{sec:general-index}, we give an explicit formula for the reduction into single $\wp$-functions by considering their generating function. There is an application to multiple Eisenstein series, we obtain the shuffle antipode relations for multiple Eisenstein series using only methods from complex analysis. In \Cref{sec:algebra}, we discuss the space of multiple $\wp$-functions $\MPF$. In \ref{appendix}, we discuss the Fourier expansion and modular transformation properties of some specific multiple $\wp$-functions, which are not directly related to the main content.

\section{Analytic properties of multiple $\wp$-functions}\label{sec:ell func}

In this section, we recall the standard facts from the theory of elliptic functions and provide some analytic properties of multiple $\wp$-functions.

\subsection{Elliptic functions}

In this subsection, we fix $\tau\in\HH$. Elliptic functions are doubly periodic meromorphic functions associated with a lattice $\Lambda\subset\C$. In this paper, we only consider the case $\Lambda=L_\tau=\Z\tau+\Z$, i.e. a function $f(z;\tau)$ (we often omit $\tau$) which is meromorphic with respect to $z$ is called an elliptic function if it satisfies the following functional equation:
\begin{align}
    f(z+1)=f(z+\tau)=f(z).
\end{align}
It is well known that any elliptic function can be written in terms of a rational function of $\wp$ and $\wp^\prime$, the Weierstrass $\wp$-function and its derivative with respect to $z$. Our multiple $\wp$-functions are also elliptic functions (\cref{lem:doubly-multi-p}), and one motivation of this paper is to obtain an explicit formula expressing multiple $\wp$-function in terms of the Weierstrass $\wp$-function and its (higher) derivatives.

\begin{lem}[see \cite{Lang}]\label{lem:ell func}
    \begin{enumerate}
        \item Let $f(z)$ be an elliptic function. Then, we have
        \begin{align}
            f(z)\in\C(\wp(z),\wp^\prime(z)).
        \end{align}
        \item If an elliptic function $f(z)$ is entire, then $f(z)$ is a constant. 
    \end{enumerate}
\end{lem}

Weierstrass $\wp$-function is a foundational building block. However, it belongs to a closely related family of three functions $\sigma$, $\zeta$, $\wp$. We recall (modified) Weierstrass $\sigma$-function, which is defined by
\begin{equation}\label{eq:def-sigma}
\sigma(z;\tau)
\ceq z\exp\left(-\frac{1}{2}G_2z^2\right)\prod_{0\neq w\in\Z\tau+\Z}
\left(1-\frac{z}{w}\right)\exp\left(\frac{z}{w}+\frac{1}{2}\frac{z^2}{w^2}\right)
= z\exp\left(-\sum_{n=0}^{\infty}\frac{G_{n+2}}{n+2}z^{n+2}\right).
\end{equation}
Moreover, we define $\mathcal{S}(z;\tau) \ceq \log \sigma(z;\tau)$ as a multivalued function. Then $\wp_2(z;\tau)$ can be recovered by differentiating $\mathcal{S}(z;\tau)$, i.e. we have
\begin{align}\label{eq:mathcalP}
-\frac{d^2}{dz^2}\mathcal{S}(z;\tau) = \wp_2(z;\tau).
\end{align}
For future reference, we also introduce (modified) Weierstrass $\zeta$-function, which is defined by
\begin{equation}\label{eq:Laurent-expansion-zeta}
\zeta(z;\tau)\ceq \frac{\sigma^\prime(z;\tau)}{\sigma(z;\tau)}=\frac{d}{dz}\mathcal{S}(z;\tau)=\frac{1}{z}-\sum_{n=0}^{\infty}G_{n+2}z^{n+1}.
\end{equation}
The classical $\wp$-function and $\wp_2$ can be also seen as a generating function of Eisenstein series:
\begin{align}\label{eq:Laurent-expansion-wp}
    \wp_2(z;\tau) = \wp(z;\tau)+G_2 = \frac{1}{z^2}+\sum_{k=0}^{\infty}(k+1)G_{k+2}(\tau)z^{k}.
\end{align}
And a single $\wp$-function of weight $k\geq3$ is a higher derivative of Weierstrass $\wp$-function:
\begin{align}
    \wp_k(z;\tau)=\frac{(-1)^k}{(k-1)!}\frac{d^{k-2}}{dz^{k-2}}\wp(z;\tau).
\end{align}

\subsection{Basic properties of multiple $\wp$-functions}

In this subsection, we propose some basic properties of multiple $\wp$-functions. We begin by defining multivariable $\wp$-functions and restricted multivariable $\wp$-functions as multivariable meromorphic functions, and then present several lemmas.

\begin{definition}
For an integer $r\geq 1$, integers $k_1,\dots,k_r\geq 2$, and $\tau\in\mathbb{H}$, we define a multivariable $\wp$-function and a restricted multivariable $\wp$-function respectively by
\begin{align}
\wp_{k_1,\dots,k_r}(z_1,\dots,z_r;\tau)
&\ceq \lim_{M\to\infty}\lim_{N\to\infty}\sum_{\substack{w_s\in\Z_M\tau+\Z_N\\ w_1\prec\dots\prec w_r}}\frac{1}{(z_1-w_1)^{k_1}\cdots(z_r-w_r)^{k_r}},
\\
\til{\wp}_{k_1,\dots,k_r}(z_1,\dots,z_r;\tau)
&\ceq \lim_{M\to\infty}\lim_{N\to\infty}\sum_{\substack{w_s\in\Z_M\tau+\Z_N\\ 0\prec w_1\prec\dots\prec w_r}}\frac{1}{(z_1-w_1)^{k_1}\cdots(z_r-w_r)^{k_r}}.
\end{align}
Moreover, we set $\wp_{k_1,\dots,k_r}(-)=\til{\wp}_{k_1,\dots,k_r}(-)=1$ if $r=0$.
\end{definition}
Note that $\til{\wp}_{k_1,\dots,k_r}(z_1,\dots,z_r)$ is holomorphic at $z_i=0$ for each $i$.

\begin{lem}\label{lem:def-multi-p}
For integers $k_1,\dots,k_r\geq2$, the defining series for $\til{\wp}_{k_1,\dots,k_r}(z;\tau)$ converges uniformly on any compact subsets of $(\mathbb{C}\setminus\til{L_\tau})^r$, where $\til{L_\tau}\ceq L_\tau\cap(\HH\cup\R_{>0})$, and thus $\til{\wp}_{k_1,\dots,k_r}(z;\tau)$ is a meromorphic function that has poles only at the lattice points.
\end{lem}

\begin{proof}
If $k_r\geq 3$, the series converges absolutely and uniformly on any compact subsets of $(\mathbb{C}\setminus \til{L_\tau})^r$. Hence, we consider the case $k_r=2$. In this case, we examine the following series:
\begin{align}
H_{k_1,\dots,k_{r-1}}(z_1,\dots,z_r;\tau)
&\ceq \lim_{M\to\infty}\lim_{N\to\infty}\sum_{\substack{w_s\in\Z_M\tau+\Z_N\\ 0\prec w_1\prec\dots\prec w_r}}\frac{-1}{(z_1-w_1)^{k_1}\cdots(z_{r-1}-w_{r-1})^{k_{r-1}}
(z_r-(w_r-1))(z_r-w_r)}.
\end{align}
Then, we have
\begin{align}\label{eq:H}
H_{k_1,\dots,k_{r-1}}(z_1,\dots,z_r;\tau)
&= \lim_{M\to\infty}\sum_{\substack{w_s\in\Z_M\tau+\Z\\ 0\prec w_1\prec\dots\prec w_{r-1}}}\frac{1}{(z_1-w_1)^{k_1}\cdots(z_{r-1}-w_{r-1})^{k_{r-1}}
(z_r-w_{r-1})}.
\end{align}
since the following identity holds
\begin{align}
\frac{-1}{(z_r-(w_r-1))(z_r-w_r)}
=\frac{1}{z_r-(w_r-1)}-\frac{1}{z_r-w_r}.
\end{align}
Since there exist some constant $C\in\mathbb{R}$ such that,
\begin{align}
\lim_{M\to\infty}\sum_{\substack{w_s\in\Z_M\tau+\Z\\ 0\prec w_1\prec\dots\prec w_{r-1}}}\left|\frac{1}{(z_1-w_1)^{k_1}\cdots(z_{r-1}-w_{r-1})^{k_{r-1}}
(z_r-w_{r-1})}\right|
\leq\sum_{\substack{w_s\in\Z\tau+\Z\\ 0\prec w_1\prec\dots\prec w_{r-1}}}\frac{C}{|w_1|^{k_1}\cdots|w_{r-1}|^{k_{r-1}+1}}
\end{align}
for any given compact subset of $(\mathbb{C}\setminus \til{L_\tau})^r$, the series in \eqref{eq:H} converges absolutely and uniformly on any compact subset of $(\mathbb{C}\setminus \til{L_\tau})^r$, and hence defines a meromorphic function. Furthermore, the series 
\begin{align}\label{eq:H-tilde}
\til{H}_{k_1,\dots,k_{r-1}}(z_1,\dots,z_r;\tau)
&\ceq \sum_{\substack{w_s\in\Z\tau+\Z\\ 0\prec w_1\prec\dots\prec w_r}}\frac{1}{(z_1-w_1)^{k_1}\cdots(z_{r-1}-w_{r-1})^{k_{r-1}}
(z_r-(w_r-1))(z_r-w_r)^2}
\end{align}
also converges absolutely and uniformly on any compact subset of $(\mathbb{C}\setminus \til{L_\tau})^r$. Indeed, for any compact subspace $K\subset\C\setminus\til{L_\tau}$, there exist a constant $C^\prime\in\R$ such that
\begin{align}
    \frac{1}{|z_r-(w_r-1)|}\leq\frac{C^\prime}{|z_r-w_r|}
\end{align}
for any $z_r\in K$ and $w_r\in\Z\tau+\Z$. And thus, the series in \eqref{eq:H-tilde} converges absolutely. Since 
\begin{align}\label{eq:wp tilde}
\til{\wp}_{k_1,\dots,k_{r-1},2}(z_1,\dots,z_r;\tau)=\til{H}_{k_1,\dots,k_{r-1}}(z_1,\dots,z_r;\tau)-H_{k_1,\dots,k_{r-1}}(z_1,\dots,z_r;\tau),
\end{align}
the claim follows.
\end{proof}

A (non-restricted) multivariable $\wp$-function can be written via restricted multivariable $\wp$-functions.

\begin{lem}\label{lem:p-to-p-tilde} For positive integers $k_1,\dots,k_r\geq2$, we have
\begin{align}
\wp_{k_1,\dots,k_r}(z_1,\dots,z_r)
=\sum_{i=0}^{r}&(-1)^{k_1+\cdots+k_{i}}\til{\wp}_{k_i,\dots,k_1}(-z_{i},\dots,-z_{1})\til{\wp}_{k_{i+1},\dots,k_r}(z_{i+1},\dots,z_{r})
\\
&+\sum_{i=1}^{r}\frac{1}{z_i^{k_i}}(-1)^{k_1+\cdots+k_{i-1}}\til{\wp}_{k_{i-1},\dots,k_1}(-z_{i-1},\dots,-z_{1})\til{\wp}_{k_{i+1},\dots,k_r}(z_{i+1},\dots,z_{r}).
\end{align}
And therefore, $\wp_{k_1,\dots,k_r}(z;\tau)$ is a meromorphic function on $\C^r$. In particular, $\wp_{k_1,\dots,k_r}(z;\tau)$ is a meromorphic function that has poles only at the lattice points. 
\end{lem}
\begin{proof}
From the decomposition of the summation
\begin{align}\label{eq:sum decomp}
\sum_{\substack{w_s\in\Z_M\tau+\Z_N\\ w_1\prec\dots\prec w_r}}
&= \sum_{\substack{w_s\in\Z_M\tau+\Z_N\\ 0\prec w_1\prec\dots\prec w_r}}
+\sum_{\substack{w_s\in\Z_M\tau+\Z_N\\ w_1\prec 0\prec w_2\dots\prec w_r}}
+\cdots+
\sum_{\substack{w_s\in\Z_M\tau+\Z_N\\ w_1\prec\dots\prec w_r\prec 0}}
\\
&+\sum_{\substack{w_s\in\Z_M\tau+\Z_N\\ 0=w_1\prec\dots\prec w_r}}
+\sum_{\substack{w_s\in\Z_M\tau+\Z_N\\ w_1\prec 0=w_2\dots\prec w_r}}
+\cdots+
\sum_{\substack{w_s\in\Z_M\tau+\Z_N\\ w_1\prec\dots\prec w_{r-1}\prec 0=w_r}},
\end{align}
we have the statement.
\end{proof}

These functions naturally arise as generating functions for multiple $\wp$-functions and multiple Eisenstein series. In the following, the notation $|x|\ll_{\tau}1$ means that $|x|<\ep$ for some sufficiently small $\ep>0$ depending on $\tau$. 

\begin{lem}\label{lem:p-til-expansion}
For $\tau\in\HH$ and $\boldsymbol{x}=(x_1,\dots,x_r)$ with $|x_i|\ll_{\tau} 1$ for each $i$, we have
\begin{align}\label{eq:mpf gen func}
\wp_{\{2\}^r}(z-x_1,\dots,z-x_r;\tau) &= \sum_{n_1,\dots,n_r\geq 0}(n_1+1)\cdots(n_r+1)\wp_{n_1+2,\dots,n_r+2}(z;\tau)x_1^{n_1}\cdots x_r^{n_r}.
\end{align}
Moreover, for $\tau\in\HH$ and $\boldsymbol{x}=(x_1,\dots,x_r)$ with $|x_i|\ll_{\tau} 1$ for each $i$, we have
\begin{align}\label{eq:p til exp}
\til{\wp}_{\{2\}^r}(x_1,\dots,x_r;\tau) &= \sum_{n_1,\dots,n_r\geq 0}(n_1+1)\cdots(n_r+1)\til{G}_{n_1+2,\dots,n_r+2}(\tau)x_1^{n_1}\cdots x_r^{n_r},
\end{align}
and more generally, by differentiating with respect to $x_i$'s, we have
\begin{align}
\til{\wp}_{k_1,\dots,k_r}(x_1,\dots,x_r;\tau)
= \sum_{n_1,\dots,n_r\geq 0}\prod_{j=1}^{r}(-1)^{k_j}\binom{n_j+k_j-1}{k_j-1}
\til{G}_{n_1+k_1,\dots,n_r+k_r}(\tau)x_1^{n_1}\cdots x_r^{n_r}
\end{align}
for positive integers $k_1,\dots,k_r\geq2$.
\end{lem}

\begin{proof} 
Here, we consider only the case of $\til{\wp}_{\{2\}^r}(x_1,\dots,x_r;\tau)$, but the proof of $\wp_{\{2\}^r}(z-x_1,\dots,z-x_r;\tau)$ can be carried out in the same way. Formally, we have
\begin{align}
\til{\wp}_{\{2\}^r}(x_1,\dots,x_r;\tau)
&= \lim_{M\to\infty}\lim_{N\to\infty}\sum_{\substack{w_s\in\Z_M\tau+\Z_N\\ 0\prec w_1\prec\dots\prec w_r}}\sum_{n_1,\dots,n_r\geq 0}
\frac{(n_1+1)\cdots (n_r+1)}{w_1^{n_1+2}\cdots w_r^{n_r+2}}x_1^{n_1}\cdots x_r^{n_r}
\\
&= \sum_{n_1,\dots,n_r\geq 0}(n_1+1)\cdots(n_r+1)\til{G}_{n_1+2,\dots,n_r+2}(\tau)x_1^{n_1}\cdots x_r^{n_r},
\end{align}
since the Taylor expansion
\begin{align}
\frac{1}{(x_s-w_s)^{2}}
= \frac{1}{w_s^2(1-\frac{x_s}{w_s})^{2}}
= \sum_{n_s=0}^{\infty}(n_s+1)\frac{1}{w_s^{n_s+2}}x^{n_s}.
\end{align}
Therefore, it suffices to check that the limits $N\to\infty$ and $M\to\infty$ can be interchanged with the Taylor expansions. Since the series converges absolutely with respect to $N\to\infty$, this limit can be interchanged with the Taylor expansions. Furthermore, using the relation \cref{eq:wp tilde}, we can also interchange the limit $M\to\infty$ with the Taylor expansions, since both equations \cref{eq:H,eq:H-tilde} converge absolutely.
\end{proof}

By the following lemma and \cref{lem:ell func}, we know that any multiple $\wp$-function $\wp_{k_1,\dots,k_r}(z)$ can be written as a rational function of $\wp$ and $\wp^\prime$.

\begin{lem}\label{lem:doubly-multi-p} For any lattice point $\lambda\in L_\tau$, we have
\begin{align}
\wp_{k_1,\dots,k_r}(z_1+\lambda,\dots,z_r+\lambda;\tau) &= \wp_{k_1,\dots,k_r}(z_1,\dots,z_r;\tau).
\end{align}
In particular, a multiple $\wp$-function $\wp_{k_1,\dots,k_r}(z;\tau)$ is an elliptic function.
\end{lem}

\begin{proof}
For a word $w=x_1\dots x_{r-1}$ ($x_i\in\{x,y\}$) of length $r-1$ and integers $k_1,\dots,k_r\geq 2$, we define
$$\wp_{k_1,\cdots,k_r}^w(z_1,\dots,z_r;\tau)
\ceq \lim_{M\to\infty} \lim_{N\to\infty}
\sum_{\substack{m_s\in\Z_M\tau+\Z_N\\ w_{i+1}-w_i\in P_{x_i}\\1\leq i\leq r-1}}
\frac{1}{(z_1+w_1)^{k_1}\cdots(z_r+w_r)^{k_r}},$$
where $P_{x} = \{n\in\Z\mid n>0\}$ and  $P_{y} = \{m\tau+n\in L_\tau\mid m>0,n\in\Z\}$. Then, we can check that
$$\wp_{k_r,\dots,k_1}(z_r,\dots,z_1;\tau) = \sum_{w}\wp_{k_1,\dots,k_r}^w(z_1,\dots,z_r;\tau),$$
where the summation on the right-hand side runs over all words of length $r-1$. Furthermore, for a word $$w=\underbrace{x\dots x}_{t_1-1}y\underbrace{x\dots x}_{t_2-t_1-1}y\dots
\underbrace{x\dots x}_{t_{h-1}-t_{h-2}-1}y\underbrace{x\dots x}_{r-t_{h-1}-1},$$
we can express $\wp_{k_1,\dots,k_r}^w(z_1,\dots,z_r;\tau)$ as
\begin{equation}\label{eq:pw-with-word}
\wp_{k_1,\dots,k_r}^w(z_1,\dots,z_r;\tau)
= \lim_{M\to\infty}\sum_{\substack{-M<m_1<\dots<m_{h}<M\\n_{t_{i-1}}<\dots<n_{t_i-1}\\1\leq i\leq h}}
\prod_{i=1}^{h}\prod_{j=t_{i-1}}^{t_{i}-1}\frac{1}{(z_j+m_i\tau+n_j)^{k_{j}}},
\end{equation}
where $t_0=1$ and $t_{h}=r+1$. Especially, we have
\begin{align}
\wp_{k_1,\dots,k_r}^w(z;\tau)
&= \lim_{M\to\infty}\sum_{-M<m_1<\dots<m_h<M}\Psi_{k_1,\dots,k_{t_1-1}}(z+m_1\tau)\cdots \Psi_{k_{t_{h-1}},\dots,k_{r}}(z+m_{h}\tau).
\end{align}
Now we prove that
\begin{align}
\wp^w_{k_1,\dots,k_r}(z_1+\lambda,\dots,z_r+\lambda;\tau) &= \wp^w_{k_1,\dots,k_r}(z_1,\dots,z_r;\tau)
\end{align}
for any lattice point $\lambda\in L_\tau$ and any word $w$. It suffices to prove the cases $\lambda=1$ and $\lambda=\tau$. From \cref{eq:pw-with-word}, the case $\lambda=1$ is obvious. Furthermore, we have
\begin{align}
&\wp_{k_1,\dots,k_r}^w(z_1+\tau,\dots,z_r+\tau;\tau)
-\wp_{k_1,\dots,k_r}^w(z_1,\dots,z_r;\tau)
\\
&= \lim_{M\to\infty}\Biggl(\sum_{\substack{-M+1<m_1<\dots<m_h<M+1\\n_{t_{i-1}}<\dots<n_{t_i-1}\\1\leq i\leq h}}-\sum_{\substack{-M<m_1<\dots<m_h<M\\n_{t_{i-1}}<\dots<n_{t_i-1}\\1\leq i\leq h}}\Biggr)
\prod_{i=1}^{h}\prod_{j=t_{i-1}}^{t_{i}-1}\frac{1}{(z_j+m_i\tau+n_j)^{k_{j}}}
\\
&= \lim_{M\to\infty}\Biggl(
\sum_{\substack{-M-1<m_1<\dots<m_{h-1}<M+1\\n_{t_{i-1}}<\dots<n_{t_i-1}\\1\leq i\leq h}}
-\sum_{\substack{m_1<m_2<\dots<m_h<M+1\\m_1\in\{-M,-M+1\}\\n_{t_{i-1}}<\dots<n_{t_i-1}\\1\leq i\leq h}}
-\sum_{\substack{-M<m_1<\dots<m_h<M\\n_{t_{i-1}}<\dots<n_{t_i-1}\\1\leq i\leq h}}\Biggr)
\prod_{i=1}^{h}\prod_{j=t_{i-1}}^{t_{i}-1}\frac{1}{(z_j+m_i\tau+n_j)^{k_{j}}}
\\
&=0
\end{align}
since
$$\sum_{n_1<\dots<n_{t_1-1}}
\prod_{j=1}^{t_1-1}\frac{1}{(z_j-M\tau+n_j)^{k_{j}}}\longrightarrow 0 \quad \text{as} \quad M\to \infty.$$
Hence, we have the case $\lambda=\tau$.
\end{proof}

\section{Explicit formula for the multiple $\wp$-function with repeated index}\label{sec:explicit-formula}

In this section, we provide an explicit formula expressing multiple $\wp$-functions with repeated indices in terms of single $\wp$-functions using partition traces. We recall that a partition of $r$ is any non-increasing sequence of positive integers $\lambda=(\lambda_1,\lambda_2,\dots,\lambda_s)$ that sum to $r$, denoted $\lambda\vdash r$. Equivalently, we use the notation $\lambda=(1^{m_1},2^{m_2},\dots,r^{m_r})\vdash r$, where $m_k$ is the multiplicity of $k$. The index $\{h\}^r$ denotes $\underbrace{h,\dots,h}_r$.

\begin{lem}[{\cite[Example 5.2.10]{Stan}}]\label{lem:3.1} Let $x_k$ ($k\ge1$) be variables. As a formal power series in $Y$, we have
\begin{align}
\exp\left(\sum_{k\geq 1}x_kY^k\right)
&=\sum_{r\ge0}Z(\mathfrak{S}_r)\left(x_1,2x_2,\dots,rx_r\right)Y^r,
\end{align}
where $Z(\mathfrak{S}_r)$ is a \emph{cycle index polynomial}\footnote{In \cite{Stan}, the polynomial is defined in terms of permutations. Alternatively, we can reformulate this definition using partitions (see \cite[Lemma 3.1]{AOS}).} for the symmetric group $\mathfrak{S}_r$ defined by
\begin{align}
    Z(\mathfrak{S}_r)(x_1,\dots,x_r)\ceq\summ{\lambda\vdash r\\\lambda=(1^{m_1},\dots,r^{m_r})}\prod_{k=1}^r\frac{1}{m_k!}\left(\frac{x_k}{k}\right)^{m_k}.
\end{align}
\end{lem}

\begin{lem}\label{lem:multi-to-deri} For any integers $r\geq 1$ and $h\geq 2$, we have
\begin{align}
\wp_{\{h\}^r}(z)
= (-1)^r \sum_{\substack{\lambda\vdash r\\ \lambda=(1^{m_1},\dots,r^{m_r})}}\prod_{k=1}^{r}\frac{1}{m_k!}\left(\frac{(-1)^{hk-1}h}{(hk)!}\wp_2^{(hk-2)}(z)\right)^{m_k}.
\end{align}
\end{lem}

\begin{proof}
    As mentioned in the introduction, multiple $\wp$-functions satisfy the harmonic product. By the results from Hoffman--Ihara (\cite{HI}, Corollary 5.1), it holds
    \begin{align}
    \sum_{r\ge0}(-1)^r\wp_{\{h\}^r}(z;\tau)Y^{hr}
    &=\exp\left(-\sum_{k\ge1}\frac{\wp_{hk}(z;\tau)}{k}Y^{hk}\right)
    \\
    &=\exp\left(\sum_{k\ge1}\frac{(-1)^{hk-1}h}{(hk)!}\wp_2^{(hk-2)}(z;\tau)Y^{hk}\right).
    \end{align}
    Applying \cref{lem:3.1} to $x_k=\frac{(-1)^{hk-1}h}{(hk)!}\wp_2^{(hk-2)}(z)$ and $Y=Y^h$, we have
    \begin{align}
        \exp\left(\sum_{k\ge1}\frac{(-1)^{hk-1}h}{(hk)!}\wp_2^{(hk-2)}(z;\tau)Y^{hk}\right)=\sum_{r\ge0}\summ{\lambda\vdash r\\\lambda=(1^{m_1},\dots,r^{m_r})}\prod_{k=1}^r\frac{1}{m_k!}\left(\frac{(-1)^{hk-1}h}{(hk)!}\wp_2^{(hk-2)}(z)\right)^{m_k}Y^{hr}.
    \end{align}
    Comparing the coefficients, we have the conclusion.
\end{proof}

It is well known that the even-order derivatives $\wp^{(2k)}(z)$ (for $k\geq 0$) can be expressed as polynomials in $\wp(z)$ with coefficients given by modular forms. This fact follows from a differential equation satisfied by $\wp(z)$, and can be stated as follows.

\begin{lem}\label{lem:deri-to-poly}
For a non-negative integer $k$, $\wp_2^{(2k)}(z)$ can be expressed as a polynomial of degree $k+1$ in $\wp(z)$. More precisely, $\wp_2^{(2k)}(z)$ can be described as follows:
\begin{enumerate}
\item $\wp_2^{(0)}(z) = \wp(z) + G_2$.
\item For an integer $k\geq 1$, if $\wp_2^{(2k-2)}(z)=\sum_{q=0}^k\alpha_{k,q}\wp^q(z)$, then
$$\wp_2^{(2k)}(z) = \sum_{q=1}^{k}\alpha_{k,q}(2q(2q+1)\wp^{q+1}(z)-30q(2q-1)G_{4}\wp^{q-1}(z)-140q(q-1)G_{6}\wp^{q-2}(z)).$$
\end{enumerate}
In particular, $\wp_2^{(2k)} = (2k+1)!\wp^{k+1}+(\text{polynomial of degree }k-1\text{ in }\wp)$, where the coefficient of $\wp^q$ is a modular from of weight $2(k-q+1)$ for $k>0$.
\end{lem}

\begin{example}
\begin{align}
\wp_2^{(2)} &= 2!(3\wp^{2}-15G_4)
\\
\wp_2^{(4)} &= 4!(5\wp^{3}-45G_{4}\wp-70G_{6})
\\
\wp_2^{(6)} &= 6!(7\wp^{4}-84G_4\wp^2-140G_6\wp+45G_4^2)
\\
\wp_2^{(8)} &= 8!(9\wp^5-135G_4\wp^3-225G_6\wp^2+270G_4^2\wp+495G_4G_6)
\\
\wp_2^{(10)} &= 10!(11\wp^6-198G_4\wp^4-330G_6\wp^3+693G_4^2\wp^2+1710G_4G_6\wp+700G_6^2-90G_4^3)
\end{align}
\end{example}

\begin{prop}\label{thm:wp-form-two}
For $h\ge2$ and $r\ge1$, there exist quasi-modular forms $f_{s,t}^{(h,r)}$ of weight $hr-2s-3t$ such that
$$\wp_{\{h\}^r}(z) = \sum_{\substack{s\geq 0,\ t\in\{0,1\}\\2s+3t\leq h}}f_{s,t}^{(h,r)}(\tau)\wp(z)^s(\wp'(z))^t.$$
\end{prop}

\begin{proof}
By \cref{lem:multi-to-deri,lem:deri-to-poly}, $\wp_{\{h\}^r}(z)$ can be expressed as a polynomial in $\wp(z)$ and $\wp'(z)$ whose coefficients are quasi-modular forms. Furthermore, $\wp_{\{h\}^r}(z)$ has poles of order at most $h$ only at the lattice points. Thus, $2s+3t\leq h$ follows.
\end{proof}

\begin{thm}\label{thm:genelating-function}
For any integer $h\geq 2$, we have
\begin{align}
\sum_{r\ge0} (-1)^r \wp_{\{h\}^r}(z) Y^{hr}
= \sigma(z)^{-h}\prod_{j=0}^{h-1}\sigma(z-\mu_{h}^j Y),
\end{align}
where $\mu_h=e^{2\pi i/h}$ and $\wp_{\{h\}^0}(z)=1$.
\end{thm}

\begin{proof}
By \cref{lem:multi-to-deri}, we have
\begin{align}
\sum_{r\geq0}(-1)^r\wp_{\{h\}^r}(z)Y^{hr}
= \exp\left(\sum_{k=1}^{\infty}\frac{(-1)^{hk-1}h}{(hk)!}\wp_2^{(hk-2)}(z)Y^{hk}\right).
\end{align}
By \eqref{eq:mathcalP}, we have
\begin{align}
\exp\left(\sum_{k=1}^{\infty}\frac{(-1)^{hk-1}h}{(hk)!}\wp_2^{(hk-2)}(z)Y^{hk}\right)
&= \exp\left(-h\mathcal{S}(z)
+\sum_{k=0}^{\infty}\frac{(-1)^{hk}h}{(hk)!}\mathcal{S}^{(hk)}(z)Y^{hk}\right).
\end{align}
Using the Taylor expansion $\mathcal{S}(z-\mu_h^jY)=\sum_{k\ge0}\frac{(-1)^k\mathcal{S}^{(k)}(z)}{k!}\mu_h^{jk}Y^k$, we have
\begin{align}
    \sum_{j=0}^{h-1}\mathcal{S}(z-\mu_h^jY)&=\sum_{k=0}^\infty\frac{(-1)^k\mathcal{S}^{(k)}(z)}{k!}\left(\sum_{j=0}^{h-1}\mu_h^{jk}\right)Y^k\\
    &=\sum_{k=0}^\infty\frac{(-1)^{hk}h}{(hk)!}\mathcal{S}^{(hk)}(z)Y^{hk}
\end{align}
Therefore, we have
\begin{align}
    \exp\left(-h\mathcal{S}(z)+\sum_{k=0}^{\infty}\frac{(-1)^{hk}h}{(hk)!}\mathcal{S}^{(hk)}(z)Y^{hk}\right)
    &= \exp\left(-h\mathcal{S}(z)+\sum_{j=0}^{h-1}\mathcal{S}(z-\mu_h^jY)\right)\\
    &= \sigma(z)^{-h}\prod_{j=0}^{h-1}\sigma(z-\mu_{h}^j Y).
\end{align}
\end{proof}

Now, we review the notion of partition traces. For a partition $\lambda=(1^{m_1},\dots,r^{m_r})\vdash r$, we set
$$X_\lambda\ceq\prod_{j=1}^{r}X_j^{m_j}$$
and define a partition trace of $\phi:\{\lambda\mid\lambda\vdash r,r\geq 1\}\to\mathbb{C}$ by
$$\Tr_r(\phi;X_1,\dots,X_r) \ceq \sum_{\lambda\vdash r}\phi(\lambda)X_\lambda,$$
and define $\Tr_0(\phi;-)\ceq 1$ by convention. 
Partition traces for general partition functions and variables are defined in \cite{Mat}. A particularly interesting case arises when each $X_j$ is specialized to the Eisenstein series $G_{2j}$, which is called \emph{partition Eisenstein traces}, and studied in several recent papers (\cite{AGOS,AGO,AOS,Mat}). See \cite{Mat} for more historical details.

When expressing $\wp_{\{h\}^r}$ as a linear combination of $\wp_k$ ($2\leq k\leq h$), the coefficient of $\wp_h$ can be written in terms of a partition trace.

\begin{prop}\label{thm:max-coeff-wp}
For a partition $\lambda=(1^{m_1},\dots,r^{m_r})\vdash r$, we set
$$\beta(\lambda) \ceq \prod_{k=1}^{r}\frac{1}{m_k!}\left(\frac{(-1)^{k-1}}{k}\right)^{m_k}.$$
Then, for any integers $r\geq 1$ and $h\geq 2$, we have
\begin{align}
\lim_{z\to 0}\frac{\wp_{\{h\}^r}(z;\tau)}{\wp_h(z;\tau)}
&=\Tr_{r-1}(\beta;G_h(\tau),G_{2h}(\tau),\dots,G_{h(r-1)}(\tau))\\
&= (-1)^{r-1}\sum_{\substack{\lambda\vdash r-1\\ \lambda=(1^{m_1},\dots,(r-1)^{m_{r-1}})}}\prod_{k=1}^{r-1}\frac{1}{m_k!}
\left(-\frac{G_{hk}}{k}\right)^{m_k}.
\end{align}
\end{prop}

\begin{proof}
Because $\wp_h(z)\sigma(z)^h\to 1$ as $z\to 0$, by taking the limit $z\to 0$ yields
\begin{equation}\label{eq:max-coeff-sigam}
\lim_{z\to 0}\sum_{r=0}^{\infty} (-1)^r \frac{\wp_{\{h\}^r}(z;\tau)}{\wp_h(z;\tau)} Y^{hr}
= \lim_{z\to 0}\frac{\prod_{j=0}^{h-1}\sigma(z-\mu_{h}^j Y)}{\wp_h(z)\sigma(z)^h}
= (-1)^h \prod_{j=0}^{h-1}\sigma(\mu_{h}^j Y).
\end{equation}
On the other hand, by \cref{lem:3.1}, we have
\begin{align}
\prod_{j=0}^{h-1}\sigma(\mu_{h}^j Y)
&=\prod_{j=0}^{h-1}(\mu_{h}^j Y)\exp\left(-\sum_{n=0}^{\infty}\frac{G_{n+2}}{n+2}(\mu_{h}^j Y)^{n+2}\right)
\\
&=(-1)^{h-1}Y^h\exp\left(-h\sum_{n=1}^{\infty}\frac{G_{hn}}{hn}Y^{hn}\right)
\\
&=(-1)^{h-1}Y^h\left(\sum_{r=0}^{\infty}\sum_{\lambda\vdash r}\prod_{k=1}^{r}\frac{1}{m_k!}
\left(-\frac{G_{hk}}{k}\right)^{m_k} Y^{hr}\right)
\\
&=(-1)^{h-1}\sum_{r=1}^{\infty}\sum_{\lambda\vdash r-1}\prod_{k=1}^{r-1}\frac{1}{m_k!}
\left(-\frac{G_{hk}}{k}\right)^{m_k} Y^{hr}.
\end{align}
Therefore, comparing the coefficients of $Y^{hr}$, we have the statement.
\end{proof}

In the case of $h=2$, the constant term can also be expressed using partition traces.

\begin{thm}\label{thm:wp-coeff-two-index} For any integer $r\geq 1$, let
\begin{align}
f_r(\tau) &= \Tr_{r-1}(\beta;G_2(\tau),G_4(\tau),\dots,G_{2(r-1)}(\tau)),\label{eq:f_r-explicit}
\\
g_r(\tau) &= \sum_{t=0}^{r} (-1)^{r-t}(2(r-t)-1)G_{2(r-t)}(\tau)f_{t+1}(\tau),\label{eq:g_r-explicit}
\end{align}
and $G_{0}(\tau)=-1$.
Then, we have
$$\wp_{\{2\}^r}(z;\tau) = f_r(\tau)\wp_2(z;\tau) + g_r(\tau).$$
\end{thm}

\begin{proof}
By \Cref{thm:wp-form-two}, we can write $\wp_{\{2\}^r}(z;\tau)=\widehat{f}_r(\tau)\wp_2(z;\tau)+\widehat{g}_r(\tau)$ for some $\widehat{f}_r$ and $\widehat{g}_r$. \Cref{thm:max-coeff-wp} yields that $\widehat{f}_r=f_r$. We now prove $\widehat{g}_r=g_r$. From \cref{thm:wp-form-two,eq:max-coeff-sigam}, the series
$$\sum_{r=0}^{\infty}(-1)^r \widehat{g}_r(\tau)Y^{2r}
= \sum_{r=0}^{\infty}(-1)^r (\wp_{\{2\}^r}(z)-f_r(\tau)\wp_2(z))Y^{2r}
= \frac{\sigma(z+Y)\sigma(z-Y)}{\sigma(z)^2}
+ \wp_2(z)\sigma(Y)^2$$
is independent of $z$ since $\widehat{g}(\tau)=\wp_{\{2\}^r}(z)-f_r(\tau)\wp_2(z)$ is independent of $z$. Hence, we can substitute $z=Y$ and then obtain
\begin{align}
\sum_{r=0}^{\infty}(-1)^r \widehat{g}_r(\tau)Y^{2r}
= \frac{\sigma(z+Y)\sigma(z-Y)}{\sigma(z)^2} + \wp_2(z)\sigma(Y)^2
= \wp_2(Y)\sigma(Y)^2
\end{align}
since $\sigma(0)=0$. On the other hand, we have
\begin{align}
\wp_2(Y)\sigma(Y)^2
&= \left(\frac{1}{Y^2}+\sum_{n=0}^{\infty}(2n+1)G_{2n+2}(\tau)Y^{2n}\right)
\left(\sum_{r=0}^{\infty}(-1)^{r}f_{r+1}(\tau)Y^{2r+2}\right)
\\
&= \sum_{r=0}^{\infty}(-1)^{r}f_{r+1}(\tau)Y^{2r}
+ \sum_{r=0}\sum_{k=0}^{r}(-1)^k f_{k+1}(\tau)
(2(r-k)+1)G_{2(r-k)+2}(\tau) Y^{2r+2}.
\end{align}
Therefore, comparing the coefficients of $Y^{2r}$, we obtain the result.
\end{proof}

Remark that the well-known formula 
\begin{equation}\label{eq:famous-eq}
\wp_2(y)-\wp_2(z)=\frac{\sigma(z+y)\sigma(z-y)}{\sigma(y)^2\sigma(z)^2}
\end{equation}
from the theory of elliptic functions arises in the course of the above proof. A similar argument can be applied to the case of a general repeated index $\{h\}^r$, yielding an identity involving the $\sigma$-function for each $h$. We now turn to the case $h=3$. For a partition $\lambda=(1^{m_1},\dots,r^{m_r})\vdash r$, we set $\ell(\lambda)\ceq m_1+\cdots+m_r$ and define
$$\beta'(\lambda) \ceq \frac{1}{2^{\ell(\lambda)}}\prod_{k=1}^{r}\frac{1}{m_k!}\frac{1}{k^{m_k}}.$$

\begin{prop}\label{thm:multi-wp-trip} For any positive odd integer $r$, we have
\begin{align}
\wp_{\{3\}^r}(z;\tau)
= -\Tr_{(r-1)/2}(\beta';G_6(\tau),G_{12}(\tau),G_{18}(\tau),\dots,G_{3(r-1)}(\tau))\cdot\wp_3(z;\tau).
\end{align}
\end{prop}

\begin{proof} Since $\wp_{\{3\}^r}(z)$ is an odd function, we obtain the result from \cref{thm:wp-form-two} and \cref{thm:max-coeff-wp}.
\end{proof}

\begin{cor} We have
\begin{align}
\sigma(2z)\prod_{j=0}^{2}\sigma(\mu_3^jY)
&=\sigma(z)\prod_{j=0}^{2}\sigma(z+\mu_3^jY)
-\sigma(z)\prod_{j=0}^{2}\sigma(z-\mu_3^jY).
\end{align}
\end{cor}

\begin{proof} From \cref{thm:multi-wp-trip,eq:max-coeff-sigam}, we have
\begin{align}
0&= 2\sum_{r=0:\text{odd}}^{\infty}(-1)^r (\wp_{\{3\}^r}(z)-f_r(\tau)\wp_3(z))Y^{hr}
\\
&= \sigma(z)^{-3}\prod_{j=0}^{2}\sigma(z-\mu_3^jY)
+ \wp_3(z)\prod_{j=0}^{2}\sigma(\mu_3^jY)
-\sigma(z)^{-3}\prod_{j=0}^{2}\sigma(z+\mu_3^jY)
-\wp_3(z)\prod_{j=0}^{2}\sigma(-\mu_3^jY),
\end{align}
where $f_r(\tau)$ is the modular form such that $\wp_{\{3\}^r}(z;\tau)=f_r(\tau)\wp_3(z;\tau)$. Here, from \cref{eq:famous-eq}, we have
\begin{equation}\label{eq:p-3-sigma-2-4}
-2\wp_3(y)=\wp(y)'=\lim_{z\to y}\frac{\wp_2(y)-\wp_2(z)}{y-z}
=\lim_{z\to y}\frac{1}{y-z}\frac{\sigma(z+y)\sigma(z-y)}{\sigma(y)^2\sigma(z)^2}
=-\frac{\sigma(2y)}{\sigma(y)^4}.
\end{equation}
Therefore, we have
\begin{align}
\sigma(2z)\prod_{j=0}^{2}\sigma(\mu_3^jY)
&=\sigma(z)\prod_{j=0}^{2}\sigma(z+\mu_3^jY)
-\sigma(z)\prod_{j=0}^{2}\sigma(z-\mu_3^jY).
\end{align}
\end{proof}

\begin{prop} We have
\begin{align}
&2\sigma(z)\sigma(Y)\sigma(z+Y)\sigma(z-Y)\sum_{i=0}^{2}\zeta(\mu_3^iY)\prod_{j=0}^2\sigma(\mu_3^jY)
\\
&\quad\quad = \sigma(z)^{3}\prod_{j=0}^2\sigma(Y+\mu_3^jY)
-\sigma(Y)^{3}\prod_{j=0}^2\sigma(z-\mu_3^jY)
-\sigma(Y)^{3}\prod_{j=0}^2\sigma(z+\mu_3^jY).
\end{align}
\end{prop}

\begin{proof} For even $r\geq 0$, since $\wp_{\{3\}^r}(z)$ is an even function, and from \cref{thm:wp-form-two}, there exist quasi-modular forms $f_r(\tau)$ and $g_r(\tau)$ such that
\begin{align}
\wp_{\{3\}^r}(z;\tau) = f_r(z)\wp_2(z)+g_r(\tau).
\end{align}
By \cref{thm:genelating-function}, we have
\begin{align}
2\sum_{r=0:\text{even}}^{\infty}(-1)^r f_r(\tau)Y^{3r}
&= \lim_{z\to 0}\frac{1}{\wp_2(z)\sigma(z)^{3}}
\left(\prod_{j=0}^2\sigma(z-\mu_3^jY)+\prod_{j=0}^2\sigma(z+\mu_3^jY)\right)
\\
&= \frac{d}{dz}\left.\left(\prod_{j=0}^2\sigma(z-\mu_3^jY)+\prod_{j=0}^2\sigma(z+\mu_3^jY)\right)\right|_{z=0}
= 2\sum_{i=0}^{2}\zeta(\mu_3^iY)\prod_{j=0}^2\sigma(\mu_3^jY).
\end{align}
Therefore, the function
\begin{align}
2\sum_{r=0:\text{even}}^{\infty}&(-1)^r(\wp_{\{3\}^r}(z;\tau)-f_r(\tau)\wp_2(z;\tau))Y^{3r}
\\
&=\frac{1}{\sigma(z)^{3}}
\left(\prod_{j=0}^2\sigma(z-\mu_3^jY)+\prod_{j=0}^2\sigma(z+\mu_3^jY)\right)
-2\wp_2(z)\sum_{i=0}^{2}\zeta(\mu_3^iY)\prod_{j=0}^2\sigma(\mu_3^jY)
\end{align}
is independent of $z$. Hence, by substituting $z=Y$, we have
\begin{align}
&\frac{1}{\sigma(z)^{3}}
\left(\prod_{j=0}^2\sigma(z-\mu_3^jY)+\prod_{j=0}^2\sigma(z+\mu_3^jY)\right)
-2\wp_2(z)\sum_{i=0}^{2}\zeta(\mu_3^iY)\prod_{j=0}^2\sigma(\mu_3^jY)
\\
&\quad\quad =\frac{1}{\sigma(Y)^{3}}
\left(\prod_{j=0}^2\sigma(Y+\mu_3^jY)\right)
-2\wp_2(Y)\sum_{i=0}^{2}\zeta(\mu_3^iY)\prod_{j=0}^2\sigma(\mu_3^jY).
\end{align}
Therefore, by \cref{eq:famous-eq}, we have
\begin{align}
&2\sigma(z)\sigma(Y)\sigma(z+Y)\sigma(z-Y)\sum_{i=0}^{2}\zeta(\mu_3^iY)\prod_{j=0}^2\sigma(\mu_3^jY)
\\
&\quad\quad = \sigma(z)^{3}\prod_{j=0}^2\sigma(Y+\mu_3^jY)
-\sigma(Y)^{3}\prod_{j=0}^2\sigma(z-\mu_3^jY)
-\sigma(Y)^{3}\prod_{j=0}^2\sigma(z+\mu_3^jY).
\end{align}
\end{proof}

For the above result, Professor Toshiki Matsusaka pointed out a non-recursive explicit formula for the higher derivatives of $\wp(z)$ as a polynomial in $\wp(z)$, using the partition trace of $\phi_{\log}$, which we present below. For a partition $\lambda=(1^{m_1},\dots,r^{m_r})\vdash r$, we set $\ell(\lambda)\ceq m_1+\cdots+m_r$ and
$$\phi_{\log}(\lambda)\ceq(\ell(\lambda)-1)!\prod_{k=1}^{r}\frac{1}{m_k!}.$$
Then, we have
$$\log\left(1-\sum_{r=1}^{\infty}x_rY^r\right)=\sum_{r=1}^{\infty}\Tr_r(\phi_{\log};x_1,\dots,x_r)Y^r
\ \quad (\text{Girard--Newton formula}).$$

\begin{thm}
For any positive integer $k$, we have
\begin{align}
\wp_2^{(2k-2)}(z) = -k(2k-1)!\Tr_k(\phi_{\log};\wp_{2}(z),-\wp_{2,2}(z),\wp_{2,2,2}(z),\dots,(-1)^{k+1}\wp_{\{2\}^k}(z)).
\end{align}
and
\begin{align}
\wp_2^{(2k-2)}(z) = (2k-1)!(G_{2k}+k\cdot \Tr_k(\phi_{\log};\wp(z),-3G_4,-5G_6,\dots,-(2k-1)G_{2k})).
\end{align}
\end{thm}

\begin{proof}
By \cref{thm:genelating-function} case $h=2$, we have
\begin{align}
1-\sum_{k=0}^{\infty} (-1)^{r-1}\wp_{\{2\}^k}(z) Y^{k}
=\exp\left(-\sum_{k=1}^{\infty} \frac{1}{k(2k-1)!}\wp_2^{(2k-2)}(z)Y^{k}\right).
\end{align}
Taking the $\log$ of both sides and comparing the coefficients of $Y^{k}$, we have
\begin{align}
\wp_2^{(2k-2)}(z) = -k(2k-1)!\Tr_k(\phi_{\log};\wp_{2}(z),-\wp_{2,2}(z),\wp_{2,2,2}(z),\dots,(-1)^{k+1}\wp_{\{2\}^k}(z)).
\end{align}
Moreover, by \cref{thm:genelating-function} case $h=2$, using \cref{eq:Laurent-expansion-wp}, \cref{eq:def-sigma}, and \cref{eq:famous-eq}, we have
\begin{align}
\exp\left(\sum_{k=1}^{\infty} \frac{1}{k}\left(G_{2k}-\frac{\wp_2^{(2k-2)}(z)}{(2k-1)!}\right) Y^{2k}\right)
&= Y^2\frac{\sigma(z+Y)\sigma(z-Y)}{\sigma(z)^2\sigma(Y)^2}
= Y^2(\wp_2(Y)-\wp_2(z))
\\
&=1-\wp(z)Y^2+\sum_{k=2}^{\infty}(2k-1)G_{2k}Y^{2k}.
\end{align}
Taking the $\log$ of both sides and comparing the coefficients of $Y^{2k}$, we have
\begin{align}
\frac{1}{k}\left(\frac{\wp_2^{(2k-2)}(z)}{(2k-1)!}-G_{2k}\right)
= \Tr_k(\phi_{\log};\wp(z),-3G_4,-5G_6,\dots,-(2k-1)G_{2k}).
\end{align}
Therefore, we obtain
\begin{align}
\wp_2^{(2k-2)}(z) = (2k-1)!(G_{2k}+k\cdot \Tr_k(\phi_{\log};\wp(z),-3G_4,-5G_6,\dots,-(2k-1)G_{2k})),
\end{align}
and this completes the proof of the statement.
\end{proof}

\section{Multiple $\wp$-function with general index}\label{sec:general-index}

In this section, we obtain expressions for multiple $\wp$-functions with general index in terms of single $\wp$-functions and multiple Eisenstein series. 

\begin{thm}\label{thm:multi-wp-general-index}
For $\boldsymbol{x}=(x_1,\dots,x_r)$ and positive integers $1\leq j\leq r$, we set
\begin{align}
\mathcal{Q}_{r,i}(z;\tau;\boldsymbol{x})=\til{\wp}_{\{2\}^{r-i}}(z-x_{i+1},\dots,z-x_r;\tau)\til{\wp}_{\{2\}^{i-1}}(x_{i-1}-z,\dots,x_1-z;\tau).
\end{align}
Then, for $\tau\in\HH$ and $\boldsymbol{x}=(x_1,\dots,x_r)$ with $|x_i|\ll_{\tau} 1$ for each $i$, we have
\begin{equation}\label{eq:multi-wp-general-index}
\wp_{\{2\}^{r}}(z-x_1,\dots,z-x_r;\tau)
=\sum_{i=1}^{r}\frac{d}{dx_i}\left[(\zeta(z-x_i)+\zeta(x_i))\mathcal{Q}_{r,i}(x_i;\tau;\boldsymbol{x})\right]
+\wp_{\{2\}^{r}}(x_r,\dots,x_1;\tau).
\end{equation}
\end{thm}

\begin{proof}
We know that $\wp_{\{2\}^{r}}(z-x_1,\dots,z-x_r;\tau)-\sum_{i=1}^{r}\frac{1}{(z-x_i)^2}\mathcal{Q}_{r,i}(z;\tau;\boldsymbol{x})$ is holomorphic at $z=x_i$ for each $i$ by considering the decomposition of the summation as same as \eqref{eq:sum decomp}. Therefore, the function
\begin{equation}\label{eq:def-F}
F_r(z;\tau;\boldsymbol{x})\ceq \wp_{\{2\}^{r}}(z-x_1,\dots,z-x_r;\tau)
-\sum_{i=1}^{r}\wp_2(z-x_i)\mathcal{Q}_{r,i}(x_i;\tau;\boldsymbol{x})
-\sum_{i=1}^{r}\zeta(z-x_i)\frac{d}{dx_i}\mathcal{Q}_{r,i}(x_i;\tau;\boldsymbol{x})
\end{equation}
is holomorphic on $z=x_i$ for each $i$, as follows from the fact that the principal parts in the Laurent expansions in \cref{eq:Laurent-expansion-wp,eq:Laurent-expansion-zeta} are $1/z^2$ and $1/z$, respectively. Then, for any $w\in\Z\tau+\Z$, by the double periodicity of $\wp_{\{2\}^{r}}(z-x_1,\dots,z-x_r;\tau)$ as a function of $z$, we have
\begin{align}\label{eq:F-eta}
F_r(z+w;\tau;\boldsymbol{x})-F_r(z;\tau;\boldsymbol{x}) = -\eta_w\sum_{i=1}^{r}\frac{d}{dx_i}\mathcal{Q}_{r,i}(x_i;\tau;\boldsymbol{x}),
\end{align}
where $\eta_w=\zeta(z+w)-\zeta(z)$. Remark that $\eta_w$ is independent of $z$. Hence, $\frac{d}{dz}F_r(z;\tau;\boldsymbol{x})$ is a doubly periodic function and holomorphic on $\mathbb{C}$, and therefore $F_r(z;\tau;\boldsymbol{x})$ can be expressed as $F_r(z;\tau;\boldsymbol{x})=az+b$ for some $a,b\in\mathbb{C}$. 
Therefore, we have $\tau(F_r(z+1;\tau;\boldsymbol{x})-F_r(z;\tau;\boldsymbol{x}))=F_r(z+\tau;\tau;\boldsymbol{x})-F_r(z;\tau;\boldsymbol{x})$ and, by the Equation \cref{eq:F-eta}, this equation is equivalent to the following:
\begin{align}
-\tau\eta_1\sum_{i=1}^{r}\frac{d}{dx_i}\mathcal{Q}_{r,i}(x_i;\tau;\boldsymbol{x})=-\eta_\tau\sum_{i=1}^{r}\frac{d}{dx_i}\mathcal{Q}_{r,i}(x_i;\tau;\boldsymbol{x}).
\end{align}
It is known that $\eta_1\tau-\eta_\tau=2\pi i$, which is known as the Legendre relation (see \cite{Lang}). From this formula, we have
\begin{align}\label{eq:antipode-Q}
\sum_{i=1}^{r}\frac{d}{dx_i}\mathcal{Q}_{r,i}(x_i;\tau;\boldsymbol{x})=0.
\end{align}
Hence, we have $F_r(z;\tau;\boldsymbol{x})=b$, which implies that $F_r(z;\tau;\boldsymbol{x})$ is independent of $z$. Therefore, by substituting $z=0$, we have
\begin{equation}\label{eq:const-F}
F_r(-;\tau;\boldsymbol{x})
=\wp_{\{2\}^{r}}(x_r,\dots,x_1;\tau)
-\sum_{i=1}^{r}\wp_2(x_i)\mathcal{Q}_{r,i}(x_i;\tau;\boldsymbol{x})
+\sum_{i=1}^{r}\zeta(x_i)\frac{d}{dx_i}\mathcal{Q}_{r,i}(x_i;\tau;\boldsymbol{x}).
\end{equation}
Therefore, by combining and rearranging \cref{eq:def-F,eq:const-F}, we obtain the theorem.
\end{proof}

\begin{cor}\label{cor:antipode-p}
For a positive integer $r$, we have
\begin{align}
\sum_{i=1}^{r}\frac{d}{dx_i}\til{\wp}_{\{2\}^{r-i}}(x_i-x_{i+1},\dots,x_i-x_r;\tau)
\til{\wp}_{\{2\}^{i-1}}(x_{i-1}-x_i,\dots,x_1-x_i;\tau)=0.
\end{align}
\end{cor}

\begin{proof}
This is exactly \cref{eq:antipode-Q}.
\end{proof}

This identity is equivalent to the shuffle antipode relation among multiple Eisenstein series, which is a family of algebraic relations induced from the antipode of the shuffle Hopf algebra. These relations are also derived from the results of Bachmann--Tasaka \cite{BT}, the construction of shuffle regularized multiple Eisenstein series.

In the following formulas, multiple Eisenstein series with indices formally less than 2 may occasionally appear. However, note that such terms do not actually appear, since the corresponding binomial coefficients vanish.

\begin{cor}\label{cor:shuffle-antipode}
For integers $k_1,\dots,k_r\ge 2$ with $k_1+\cdots+k_r=k$, we have
\begin{align}
\sum_{i=1}^r\sum_{\substack{n_1,\dots,n_r\geq 0\\n_1+\cdots+n_r=k\\n_i=1}}(-1)^{k_i+n_i+\cdots+n_r}\prod_{\substack{p=1\\p\ne i}}^r\binom{n_p-1}{k_p-1}\til{G}_{n_{i-1},\dots,n_1}(\tau)\til{G}_{n_{i+1},\dots,n_r}(\tau)=0.
\end{align}
\end{cor}

\begin{proof}
By expanding the left-hand side of the equation in \cref{cor:antipode-p} into a Laurent series using \cref{lem:p-til-expansion} and comparing the coefficients, we obtain the desired result.
\end{proof}

From \cref{thm:multi-wp-general-index}, we obtain an explicit expression of the multiple $\wp$-function in terms of the single $\wp$-functions.

\begin{thm}[Restatement of \cref{thm:rough-state-coeff-two}]\label{thm:depth-2-wp-formula}
For integers $k_1,\dots,k_r\geq 2$ with $k_1+\cdots+k_r=k$, we have
{\footnotesize
\begin{equation}\label{eq:explicit reduction}
\begin{split}
\wp_{k_1,\dots,k_r}(z;\tau)
=&\sum_{i=1}^{r}\sum_{\substack{n_1,\dots,n_r\geq 2\\n_1+\cdots+n_r=k}}(-1)^{k_i+n_{i}+\cdots+n_r}\prod_{\substack{j=1\\j\neq i}}^{r}\binom{n_j-1}{k_j-1}
\til{G}_{n_{i-1},\dots,n_1}(\tau)\til{G}_{n_{i+1},\dots,n_r}(\tau)(\wp_{n_i}(z;\tau)-G_{n_i}(\tau))
\\
&+\sum_{i=1}^{r}\sum_{\substack{n_1,\dots,n_r\geq 0\\n_1+\cdots+n_r=k\\n_i=0}}(-1)^{k_i+n_{i+1}+\cdots+n_r}\prod_{\substack{j=1\\j\neq i}}^{r}\binom{n_j-1}{k_j-1}
\til{G}_{n_{i-1},\dots,n_1}(\tau)\til{G}_{n_{i+1},\dots,n_r}(\tau)
\\
&+\sum_{i=0}^{r}(-1)^{k_{i+1}+\cdots+k_r}\til{G}_{k_{i},\dots,k_1}(\tau)\til{G}_{k_{i+1},\dots,k_r}(\tau).
\end{split}\end{equation}
}
\end{thm}

\begin{proof}
By \cref{lem:p-til-expansion} \cref{eq:mpf gen func}, the left-hand side of \cref{eq:multi-wp-general-index} can be expanded as follows.
\begin{align}
    \wp_{\{2\}^r}(z-x_1,\dots,z-x_r;\tau) &= \sum_{k_1,\dots,k_r\geq2}(k_1-1)\cdots(k_r-1)\wp_{k_1,\dots,k_r}(z;\tau)x_1^{k_1-2}\cdots x_r^{k_r-2}.
\end{align}
On the other hand, the right-hand side of \cref{eq:multi-wp-general-index} can be written as follows.
\begin{align}\label{eq:three terms}
    \sum_{i=1}^r\frac{d}{dx_i}\lbrack\zeta(z-x_i;\tau)+\zeta(x_i;\tau)\rbrack\mathcal{Q}_{r,i}(x_i;\tau;\boldsymbol{x})+\sum_{i=1}^r\left(\zeta(z-x_i;\tau)+\zeta(x_i;\tau)\right)\frac{d}{dx_i}\mathcal{Q}_{r,i}(x_i;\tau;\boldsymbol{x})+\wp_{\{2\}^{r}}(x_r,\dots,x_1;\tau).
\end{align}
It suffices to show that the coefficient of $(k_1-1)\cdots(k_r-1)x_1^{k_1-2}\cdots x_r^{k_r-2}$ in \cref{eq:three terms} coincides with the right-hand side of \cref{eq:explicit reduction}. First, we have
\begin{align}
    \zeta(z-x_i)+\zeta(x_i)&=\frac{1}{x_i}+\sum_{k_i\geq2}(\wp_{k_i-1}(z;\tau)-G_{k_i-1})x_i^{k_i-2},\\
    \mathcal{Q}_{r,i}(x_i;\tau;\boldsymbol{x})&=\sum_{k_1,\dots,k_{i-1},k_{i+1},\dots,k_r\geq2}(-1)^{k_1+\cdots+k_{i-1}}\til{G}_{k_{i-1},\dots,k_1}\til{G}_{k_{i+1},\dots,k_r}\prod_{\substack{j=1\\j\neq i}}^r(k_j-1)(x_i-x_j)^{k_j-2}
\end{align}
where we denote $\wp_1(z;\tau)=\zeta(z;\tau)$, and we used \cref{eq:p til exp} for the later expansion. Using these expansion, $\sum_{i=1}^r\frac{d}{dx_i}\lbrack\zeta(z-x_i;\tau)+\zeta(x_i;\tau)\rbrack\mathcal{Q}_{r,i}(x_i;\tau;\boldsymbol{x})$ can be written as follows:
{\footnotesize
\begin{align}
    &-\sum_{i=1}^r\summ{n_1,\dots,n_{i-1},n_{i+1},\dots,n_r\geq2\\k_1,\dots,k_{i-1},k_{i+1},\dots,k_r\geq2}(-1)^{\widehat{n}+k_1+\cdots+k_{i-1}}\til{G}_{k_{i-1},\dots,k_1}\til{G}_{k_{i+1},\dots,k_r}x_i^{\widehat{k}-\widehat{n}-2}\prod_{\substack{j=1\\j\neq i}}^r(n_j-1)\binom{k_j-1}{n_j-1}x_j^{n_j-2}\\
    &+\sum_{i=1}^r\summ{n_1,\dots,n_{i-1},n_{i+1},\dots,n_r\geq2\\k_1,\dots,k_r\geq2}(-1)^{\widehat{n}+k_1+\cdots+k_{i-1}}(k_i-1)\til{G}_{k_{i-1},\dots,k_1}\til{G}_{k_{i+1},\dots,k_r}(\wp_{k_i}(z)-G_{k_i})x_i^{k-\widehat{n}-2}\prod_{\substack{j=1\\j\neq i}}^r(n_j-1)\binom{k_j-1}{n_j-1}x_j^{n_j-2},\label{eq:first term}
\end{align}
}
where $\widehat{n}=n_1+\cdots+n_{i-1}+n_{i+1}+\cdots+n_r$, $\widehat{k}=k_1+\cdots+k_{i-1}+k_{i+1}+\cdots+k_r$ and $k=k_1+\cdots+k_r$. Also $\sum_{i=1}^r\left(\zeta(z-x_i;\tau)+\zeta(x_i;\tau)\right)\frac{d}{dx_i}\mathcal{Q}_{r,i}(x_i;\tau;\boldsymbol{x})$ can be computed as follows:
{\footnotesize
\begin{align}
    &\sum_{i=1}^r\frac{1}{x_i}\sum_{k_1,\dots,k_{i-1},k_{i+1},\dots,k_r\geq2}(-1)^{k_1+\cdots+k_{i-1}}\til{G}_{k_{i-1},\dots,k_1}\til{G}_{k_{i+1},\dots,k_r}\summ{p=1\\p\neq i}^r(k_p-1)(k_p-2)(x_i-x_p)^{k_p-3}\prod_{\substack{j=1\\j\neq i,p}}^r(k_j-1)(x_i-x_j)^{k_j-2}\\
    &+\sum_{i=1}^r\sum_{k_1,\dots,k_r\geq2}\!(-1)^{k_1+\cdots+k_{i-1}}\!(\wp_{k_i-1}(z)-G_{k_i-1})\til{G}_{k_{i-1},\dots,k_1}\til{G}_{k_{i+1},\dots,k_r}x_i^{k_i-2}\summ{p=1\\p\neq i}^r(k_p-1)(k_p-2)(x_i-x_p)^{k_p-3}\prod_{\substack{j-1\\j\neq i,p}}^r(k_j-1)(x_i-x_j)^{k_j-2}.
\end{align}
}
Note that the term for $k_i=2$ in the second sum vanishes since we have $\sum_{i=1}^r\frac{d}{dx_i}\mathcal{Q}_{r,i}(x_i;\tau;\boldsymbol{x})=0$ \cref{eq:antipode-Q}. Thus, this equals to
{\footnotesize
\begin{align}
    &\sum_{i=1}^r\summ{n_1,\dots,n_{i-1},n_{i+1},\dots,n_r\geq2\\k_1,\dots,k_{i-1},k_{i+1},\dots,k_r\geq2}(-1)^{\widehat{n}+k_1+\cdots+k_{i-1}}\til{G}_{k_{i-1},\dots,k_1}\til{G}_{k_{i+1},\dots,k_r}(\widehat{k}-\widehat{n})x_i^{\widehat{k}-\widehat{n}-2}\prod_{\substack{j=1\\j\neq i}}^r(n_j-1)\binom{k_j-1}{n_j-1}x_j^{n_j-2}\\
    &+\sum_{i=1}^r\summ{n_1,\dots,n_{i-1},n_{i+1},\dots,n_r\geq2\\k_1,\dots,k_r\geq2}(-1)^{\widehat{n}+k_1+\cdots+k_{i-1}}\til{G}_{k_{i-1},\dots,k_1}\til{G}_{k_{i+1},\dots,k_r}(\wp_{k_i}(z)-G_{k_i})(\widehat{k}-\widehat{n})x_i^{k-\widehat{n}-2}\prod_{\substack{j=1\\j\neq i}}^r(n_j-1)\binom{k_j-1}{n_j-1}x_j^{n_j-2}.
    \label{eq:second term}
\end{align}
}
Combining \cref{eq:first term}, \cref{eq:second term} and \cref{eq:p til exp}, we know that the coefficient of $(k_1-1)\cdots(k_r-1)x_1^{k_1-2}\cdots x_r^{k_r-2}$ in \cref{eq:three terms} coincides with the right-hand side of \cref{eq:explicit reduction}.
\end{proof}

In the Laurent expansion of \cref{thm:depth-2-wp-formula}, the coefficients of $z^m$ for $m\leq 0$ yield only trivial relations, whereas the coefficients of $z^m$ for $m>0$ produce non-trivial relations among multiple Eisenstein series. The following result presents relations obtained by comparing the coefficients of $z^m$ with $m>0$. Numerically, we expect that these relations can be obtained by \cref{cor:shuffle-antipode} and the harmonic product.

\begin{cor}\label{cor:Eisen-relation-from-p} For any integers $k_1,\dots,k_r\geq 2$, and $m>0$, we have
\begin{align}
\sum_{i=1}^{r}\sum_{\substack{n_1,\dots,n_r\geq 0\\n_1+\cdots+n_r=m+k}}&(-1)^{k_i+n_{i+1}+\cdots+n_r}
\binom{n_i-1}{m}\prod_{\substack{j=1\\j\neq i}}^{r}\binom{n_j-1}{k_j-1}
\til{G}_{n_{i-1},\dots,n_1}(\tau)\til{G}_{n_{i+1},\dots,n_r}(\tau)G_{n_i}(\tau)
\\
= \sum_{i=0}^{r}&\sum_{\substack{n_1,\dots,n_r\geq 0\\n_1+\dots+n_r=m+k}}
(-1)^{n_{i+1}+\cdots+n_r+m}
\prod_{j=1}^{r}\binom{n_j-1}{k_j-1}
\til{G}_{n_{i},\dots,n_1}(\tau)\til{G}_{n_{i+1},\dots,n_r}(\tau)
\\
+& \sum_{i=1}^{r}\sum_{\substack{n_1,\dots,n_r\geq 0\\n_1+\dots+n_r=m+k\\n_i=0}}
(-1)^{k_i+n_{i+1}+\cdots+n_{r}+m}
\prod_{\substack{j=1\\j\neq i}}^{r}\binom{n_j-1}{k_j-1}
\til{G}_{n_{i-1},\dots,n_1}(\tau)\til{G}_{n_{i+1},\dots,n_r}(\tau).
\end{align}
\end{cor}

From \cref{thm:depth-2-wp-formula}, one naturally obtains relations among multiple zeta values. This formula is a special case of Equation (2.2) in \cite{Hirose2}, which was derived by comparing the constant terms in the Laurent expansions of multitangent functions. The following proof essentially employs the same method as that one. Using this formula, Hirose obtained in \cite{Hirose2} an explicit expression for the parity theorem of multiple zeta values. Here, we remark that the left-hand side of the following equation is not a ``symmetric multiple zeta value''.

\begin{cor}[{\cite[Equation (2.2)]{Hirose2}}]\label{cor:MZVs-relations}
For integers $k_1,\dots,k_r\geq 2$ with $k_1+\cdots+k_r=k$, we have
\begin{align}
\sum_{i=0}^{r}&(-1)^{k_{i+1}+\cdots+k_r}\zeta(k_{i},\dots,k_1)\zeta(k_{i+1},\dots,k_r)
\\
&\quad\quad = -\sum_{i=1}^{r}\sum_{\substack{n_1,\dots,n_r\ge 0\\n_1+\cdots+n_r=k\\n_i:\text{even}}}(-1)^{k_i+n_{i+1}+\cdots+n_r}\prod_{\substack{j=1\\j\neq i}}^{r}\binom{n_j-1}{k_j-1}
\frac{(2\pi i)^{n_i}B_{n_i}}{n_i!}\zeta(n_{i-1},\dots,n_1)\zeta(n_{i+1},\dots,n_r).
\end{align}
Here, $B_k$ is the $k$-th Bernoulli number.
\end{cor}

\begin{proof}
Taking the limit $\tau\to\infty$ in \cref{thm:depth-2-wp-formula}, and using the limits
$$\wp_{k_1,\dots,k_r}(z;\tau)\to\Psi_{k_r,\dots,k_1}(z)\quad
\text{and}\quad  \til{G}_{k_1,\dots,k_r}(\tau)\to\zeta(k_1,\dots,k_r),$$
we obtain a formula expressing the multitangent function $\Psi_{k_r,\dots,k_1}(z)$ in terms of monotangent functions $\Psi_{k}(z)$ for $k\geq 2$ and multiple zeta values. Comparing this formula with the identity
\begin{align}
\Psi_{k_r,\dots,k_1}(z)
=&\sum_{i=1}^{r}\sum_{\substack{n_1,\dots,n_r\geq 2\\n_1+\cdots+n_r=k}}(-1)^{k_i+n_{i}+\cdots+n_r}\prod_{\substack{j=1\\j\neq i}}^{r}\binom{n_j-1}{k_j-1}
\zeta(n_{i-1},\dots,n_1)\zeta(n_{i+1},\dots,n_r)\Psi_{n_i}(z)
\end{align}
given in \cite[Theorem 3]{Bouillot}, we obtain the desired claim.
\end{proof}

\section{The algebra of multiple $\wp$-functions}\label{sec:algebra}
To proceed with the discussion, we define several vector spaces. For an integer $k>0$, we denote
\begin{gather}
\MES_k\ceq\Span_\Q\{\til{G}_{k_1,\dots,k_r}(\tau)\mid r\geq 1,k_1,\dots,k_r\ge2,k_1+\cdots+k_r=k\}\ ,\quad \MES_0=\Q,
\intertext{and}
\MPF_k\ceq \Span_\Q\{\wp_{k_1,\dots,k_r}(z;\tau)\mid r\geq 1,k_1,\dots,k_r\ge2,k_1+\cdots+k_r=k\} ,\quad \MPF_0=\Q.
\end{gather}
Furthermore, we define
\begin{align}
\MES\ceq\sum_{k\ge 0}\MES_k
\quad\quad\text{and}\quad\quad
\MPF\ceq\sum_{k\ge0}\MPF_k.
\end{align}
The harmonic product is defined on both $\MPF$ and $\MES$, so we have $\MES_r\cdot\MES_s\subset\MES_{r+s}$ and $\MPF_r\cdot\MPF_s\subset\MPF_{r+s}$. Moreover, we have $\displaystyle{\frac{d}{dz}}(\MPF_k)\subset\MPF_{k+1}$.

\subsection{Relations among multiple Eisenstein series from multiple $\wp$-functions}

\begin{conj}[Combining Conjecture 1 in \cite{Oko} and Theorem 1 in \cite{BK}, see also Conjecture 4.18 (2) in \cite{BIM}]\label{conj:dim-Eisen} We have
\begin{align}
\sum_{k\geq 0}\dim \MES_k X^{k} = \frac{1}{1-X^2-X^3-X^4-X^5+X^8+X^9+X^{10}+X^{11}+X^{12}}.
\end{align}
\end{conj}

The following table shows the number of independent relations $\#\text{rel}_{\text{conj}}$ predicted by the dimension conjecture, the number of independent relations $\#\text{rel}_{\text{anti}}$ obtained from the shuffle antipode (\cref{cor:shuffle-antipode}), after decomposing it via the harmonic product and taking its harmonic products with multiple Eisenstein series, and the difference $\#\text{rel}_{\text{conj}}-\#\text{rel}_{\text{anti}}$.
\begin{align}
\begin{array}{c|cccccccccccccccc}
\text{weight }k & 2 & 3 & 4 & 5 & 6 & 7 & 8 & 9 & 10 & 11 & 12 & 13 & 14 & 15 & 16
\\ \hline
\dim_{\text{conj}} \MES_k & 1 & 1 & 2 & 3 & 4 & 7 & 9 & 15 & 21 & 32 & 47 & 70 & 104 & 153 & 228
\\
\text{$\#\text{rel}_{\text{conj}}$} & 0 & 0 & 0 & 0 & 1 & 1 & 4 & 6 & 13 & 23 & 42 & 74 & 129 & 224 & 382
\\
\text{$\#\text{rel}_{\text{anti}}$} & 0 & 0 & 0 & 0 & 1 & 1 & 4 & 5 & 13 & 19 & 40 & 62 & 115 & 188 & 328
\\
\text{$\#\text{rel}_{\text{conj}}$}-\text{$\#\text{rel}_{\text{anti}}$}
& 0 & 0 & 0 & 0 & 0 & 0 & 0 & 1 & 0 & 4 & 2 & 12 & 14 & 36 & 54
\end{array}
\end{align}
We do not have a concrete conjecture regarding the precise nature of the relations that constitute $\text{$\#\text{rel}_{\text{conj}}$}-\text{$\#\text{rel}_{\text{anti}}$}$. Conjecturally, differentiation with respect to $\tau$ lifts a relation in $\MES_k$ to a relation in $\MES_{k+2}$, so the natural question is whether $\text{$\#\text{rel}_{\text{conj}}$}-\text{$\#\text{rel}_{\text{anti}}$}$ is entirely generated by such lifted relations. However, we have not been able to confirm this even in weight $9$.

Furthermore, it is conjectured that the relations obtained by decomposing \cref{cor:Eisen-relation-from-p} via the harmonic product and then taking harmonic products with multiple Eisenstein series are contained in the relations obtained from \cref{cor:shuffle-antipode}.

\subsection{The algebra of multiple $\wp$-functions}

If we regard the multiple $\wp$-function as an analog of the multitangent function, the following question can be posed as an analog of the result in \cite{Hirose}.
\begin{itemize}
\item Is $\MPF$ a $\MES$-algebra?
\end{itemize}
Numerical experiments suggest that the answer to this question is ``No''. In other words, it is conjectured that $\MES\cdot\MPF\not\subset\MPF$. We reformulate the question as follows.
\begin{conj}\label{conj:derivative}
For any integers $r,s\geq 0$ we have $\displaystyle{\MES_r\cdot\frac{d}{dz}\MPF_s\subset\MPF_{r+s+1}}$.
\end{conj}
To present several statements equivalent to this conjecture, we define a certain space of multiple Eisenstein series. From \cref{thm:rough-state-coeff-two}, there exist $\mathcal{G}_n(\mathds{k})\in\MES_{k-n}$ such that
$$\wp_{\mathds{k}}(z) = \mathcal{G}_0(\mathds{k}) + \sum_{n\geq 2}\mathcal{G}_n(\mathds{k})\wp_n(z).$$
Then we define
$$\MES(k,n)\ceq \Span_\Q\{\mathcal{G}_n(k_1,\dots,k_r)\mid r\geq 1,k_1,\dots,k_r\ge2,k_1+\cdots+k_r=k\}.$$
It immediately follows that $\MES(k,n)\subset\MES_{k-n}$.
\begin{prop}[cf. {\cite[Theorem 2]{Bouillot}}]\label{prop:equiv-derivative} The following statements are equivalent to \cref{conj:derivative}.
\begin{enumerate}[(1)]
\item For any integers $k\geq 0$ and $n\geq 3$, we have $\MES_{k}\cdot\wp_n(z)\subset\MPF_{k+n}$.
\item For any integer $k\geq 0$, we have $\MES_k\cdot\wp_3(z)\subset\MPF_{k+3}$.
\item For any integers $k\geq 0$ and $n\geq 3$, we have $\MES(k+n,n)=\MES_{k}$.
\item For any integer $k\geq 0$, we have $\MES(k+3,3)=\MES_{k}$.
\end{enumerate}
\end{prop}

\begin{proof}
Similar to \cite[Theorem 2]{Bouillot}, one can show the equivalence among \Cref{conj:derivative}, (1) and (2) by using the identity $\frac{d}{dz}\wp_{n}(z) = -n \wp_{n+1}(z)$. Now, we prove the equivalence (3) $\Leftrightarrow$ (4). It is trivial that (3) $\Rightarrow$ (4). Assuming (4), for every element $f\in\MES_k$, there exists an element $g\in\MPF_{k+3}$ such that its coefficient of $\wp_3$ is $f$. For any $n\geq3$, the coefficient of $\wp_n$ in $\left(\frac{d}{dz}\right)^{n-3}g\in\MPF_{k+n}$ is an integer multiple of $f$, and thus $f\in\MES(k+n,n)$. 
Finally, we prove the equivalence (2) $\Leftrightarrow$ (4). Assuming (2), we have $\MES_k\cdot\wp_3\subset\MPF_{k+3}\subset\MES(k+3,0)+\sum_{i\geq2}\MES(k+3,i)\cdot\wp_i$. Since $\{\wp_n(z)\mid n\geq2\}$ is linearly independent over $\C$, we have $\MES_k\subset\MES(k+3,3)$ and thus (2) $\Rightarrow$ (4).
Conversely, assuming (4), for any $f\in\MES_k=\MES(k+3,3)$, there exists an element $g(z)\in\MPF_{k+3}$ such that its coefficient of $\wp_3$ is $f$. Since $\wp_n(z)$ is an odd function when $n$ is odd, we have
\begin{align}\label{eq:symmetrize}
    \frac{g(z)-g(-z)}{2}=f\cdot\wp_3+\summ{i=5\\i:\mathrm{odd}}^{k+1}f_{k-i+3}\cdot\wp_i
\end{align}
for some $f_{k-i+3}\in\MES_{k-i+3}$. Note that the left-hand side of \cref{eq:symmetrize} is contained in $\MPF_{k+3}$ since $\wp_{k_1,\dots,k_r}(-z)=(-1)^{k_1,\dots,k_r}\wp_{k_r,\dots,k_1}(z)$. We prove $f\cdot\wp_3\in\MPF_{k+3}$ by induction on $k\geq0$. When $0\leq k\leq 3$, the sum on the right-hand side of \cref{eq:symmetrize} vanishes and thus $f\cdot\wp_3\in\MPF_{k+3}$. When $k>3$, by the inductive hypothesis, $f_{k-i+3}\cdot\wp_i\in\left(\frac{d}{dz}\right)^{i-3}\MES_{k-i+3}\cdot\wp_3$ is contained in $\left(\frac{d}{dz}\right)^{i-3}\MPF_{k-i+6}\subset\MPF_{k+3}$, which completes the proof.
\end{proof}

Based on several numerical calculations, it is conjectured that \cref{conj:derivative} holds in the formal setting. We describe this in more detail. Let $W$ be the set of words generated by $\{z_{k}\mid k\geq 2\}$, and let $\mathfrak{H}^2$ be the $\Q$-vector space generated by W. Let $\mathfrak{H}^2_*$ be the $\Q$-algebra whose commutative product $*$ is defined by
$$w*1=1*w=w\quad \text{and}\quad wz_r*vz_s = (w*vz_s)z_r + (wz_r*v)z_s + (w*v)z_{r+s}$$
for all words $w,v\in W$ and all integers $r,s\geq 2$. We define
\begin{align}
\mathfrak{Z}_a(k_1,\dots,k_r) \ceq \sum_{i=1}^{r}\sum_{\substack{n_1,\dots,n_r\geq 2\\n_1+\cdots+n_r=k\\n_i=a}}(-1)^{k_i+n_{i}+\cdots+n_r}\prod_{\substack{j=1\\j\neq i}}^{r}\binom{n_j-1}{k_j-1}z_{n_{i-1}}\cdots z_{n_1} * z_{n_{i+1}}\cdots z_{n_r},\quad (k_1+\cdots+k_r=k).
\end{align}
Then, we conjecture that for any integers $k \geq 0$ and $n \geq 3$, we have
\begin{align}
\Span_{\Q}\{&\mathfrak{Z}_n(k_1,\dots,k_r) \mid k_1,\dots,k_r\geq 2,\,k_1+\dots+k_r=k\}
\\
&= \Span_{\Q}\{z_{k_1}\dots z_{k_r} \mid k_1,\dots,k_r\geq 2,\,k_1+\dots+k_r=k-n\}.
\end{align}
The map $\mathfrak{H}^2_*\to \MES\;;\;z_{k_1}\cdots z_{k_r} \mapsto G_{k_1,\dots,k_r}(\tau)$, is a homomorphism that sends the left-hand side of the above equation to $\MES(k,n)$ and the right-hand side to $\MES_{k-n}$. Hence, this conjecture implies (3) in \cref{prop:equiv-derivative}.

\begin{conj}\label{conj:cusp-form}
Let $\MES_k=\{0\}$ for a negative integer $k$, and let $\Delta(\tau)$ be the cusp form of weight 12 on $\SL_2(\Z)$, defined by
$$\Delta(\tau) \ceq (2\pi)^{12}\prod_{n=1}^{\infty}(1-q^n)^{24}
\quad (q=e^{2\pi i\tau}).$$
Then we have the following.
\begin{enumerate}[(1)]
\item $\MPF_k\cap\MES_k=\Delta\cdot\MES_{k-12}$.
\item $\MPF_k\cap (\MES_{k-2}\cdot\wp_2(z))=\Delta\cdot\MES_{k-14}\cdot\wp_2(z)$.
\end{enumerate}
\end{conj}

\begin{example} We have
$$\Delta = -320(7\wp_{3,9}+21\wp_{4,8}+37\wp_{5,7}+45\wp_{6,6}+37\wp_{7,5}+21\wp_{8,4}+7\wp_{9,3}-350\wp_{3,3,3,3}).$$
\end{example}

From \cref{conj:derivative,prop:equiv-derivative}, we are led to the conjecture that $\MES(k,n)=\MES_{k-n}$ for $n\geq 3$. For $\MES(k,0)$ and $\MES(k,2)$ it is unclear how the spaces can be described, however, if \cref{conj:derivative,conj:cusp-form} are true, there exist the following isomorphisms of vector spaces.
\begin{align}
\quotient{\MPF_k}{\Delta\cdot\MES_{k-12}} \overset{?}{\cong} \MES(k,2) \oplus \bigoplus_{n\geq 0}^{k-3}\MES_{n}
\quad\text{and}\quad
\quotient{\MPF_k}{\Delta\cdot\MES_{k-14}\cdot\wp_2} \overset{?}{\cong} \MES(k,0) \oplus \bigoplus_{n\geq 0}^{k-3}\MES_{n}.
\end{align}
Furthermore, we have the following dimension formula for $\MPF_k$.
\begin{align}
\dim\MPF_k \overset{?}{=} \dim\MES_{k-12} + \dim\MES(k,2) + \sum_{n\geq 0}^{k-3}\dim\MES_{n}
\overset{?}{=} \dim\MES(k,0) + \dim\MES_{k-14} + \sum_{n\geq 0}^{k-3}\dim\MES_{n}.
\end{align}

Since the dimension conjecture for $\MES_k$ is already given in \cref{conj:dim-Eisen}, the dimension conjecture for $\MPF_k$ can be formulated once the dimension conjecture for $\MES(k,0)$ or $\MES(k,2)$ is obtained. However, at present, no such conjecture is known. Below, we present the computed values of $\dim\MPF_k$ under the assumption of the conjecture in \cite{BKM}, which describes all relations among multiple Eisenstein series.

$$\begin{array}{l|cccccccccccccccccc}
\text{weight }k & 2 & 3 & 4 & 5 & 6 & 7 & 8 & 9 & 10 & 11 & 12 & 13 & 14 & 15 & 16
\\ \hline
\dim \MPF_k & 1 & 1 & 2 & 3 & 5 & 7 & 12 & 17 & 27 & 39 & 59 & 88 & 130 & 194 & 286
\\ \hdashline
\dim\MES_k & 1 & 1 & 2 & 3 & 4 & 7 & 9 & 15 & 21 & 32 & 47 & 70 & 104 & 153 & 228
\\
\dim\MES(k,0) & 0 & 0 & 1 & 1 & 2 & 2 & 4 & 5 & 8 & 11 & 16 & 24 & 33 & 51 & 72
\\
\dim\MES(k+2,2) & 1 & 1 & 2 & 2 & 4 & 5 & 8 & 11 & 15 & 24 & 33 & 50 & 71 & 104 & 153
\end{array}$$

\begin{prob}
What sequence characterizes $\dim \MPF_k$?
\end{prob}

\appendix



\setcounter{thm}{0}
\renewcommand{\thethm}{\Alph{section}.\arabic{thm}}

\section{Some properties of $\wp_{2,\dots,2}(z;\tau)$}\label{appendix}

In this appendix, we give some additional properties of $\wp_{2,\dots,2}(z;\tau)$.
To state the Fourier expansion of $\wp_{2,\dots,2}(z;\tau)$, we define $g_{k_1,\dots,k_r}(z;\tau)$ for $k_1,\dots,k_r\ge2$ by
\begin{align}
    g_{k_1,\dots,k_r}(z;\tau)&\ceq\sum_{\substack{0<m_1<\cdots<m_r\\n_1,\dots,n_r\in\Z}}\frac{1}{(z+m_1\tau+n_1)^{k_1}\cdots(z+m_r\tau+n_r)^{k_r}}.
\end{align}

The Fourier expansion of these functions are given as follows.

\begin{lem}
    For $\Im(z)>0$, we have
    \begin{align}
        g_{k_1,\dots,k_r}(z;\tau)=\frac{(-2\pi i)^{k_1+\cdots+k_r}}{(k_1-1)!\cdots(k_r-1)!}\sum_{\substack{0<m_1<\cdots<m_r\\n_1,\dots,n_r>0}}n_1^{k_1-1}\cdots n_r^{k_r-1}\xi^{n_1+\cdots+n_r}q^{m_1n_1+\cdots+m_rn_r},
    \end{align}
    where $\xi=e^{2\pi iz}$ and $q=e^{2\pi i\tau}$.
\end{lem}

\begin{proof}
    This follows from the Lipschitz summation formula
    \begin{align}
        \sum_{n\in\Z}\frac{1}{(z+n)^k}=\frac{(-2\pi i)^k}{(k-1)!}\left(\frac{\delta_{k,1}}{2}+\sum_{n>0}n^{k-1}\xi^n\right)
    \end{align}
    for $k\ge1$ and $\Im(z)>0$.
\end{proof}

\begin{prop}
    For $k_1,\dots,k_r\ge2$ and $\Im(z)>0$, we have
    \begin{align}
        &\til{\wp}_{k_1,\dots,k_r}(-z;\tau)=(-1)^{k_1+\cdots+k_r}\sum_{h=0}^{r-1}\sum_{1\le{}t_1<\cdots<t_h\le{}r}\zeta^{(z)}(k_1,\dots,k_{t_1-1})\sum_{\substack{t_1\le{q_1}<t_2\\\vdots\\t_h\le{}q_h<r+1}}\sum_{\substack{n_{t_j}+\cdots+n_{t_{j+1}-1}\\=k_{t_j}+\cdots+k_{t_{j+1}-1}\\(1\le{}j\le{}h),n_\bullet\ge2}}\\
        &\times\prod_{j=1}^h\bigg\{(-1)^{l_j}\prod_{\substack{p=t_j\\p\ne{}q_j}}^{t_{j+1}-1}\binom{n_p-1}{k_p-1}\zeta(n_{q_j-1},\dots,n_{t_j})\zeta(n_{q_j+1},\dots,n_{t_{j+1}-1})\bigg\}g_{n_{q_1},\dots,n_{q_h}}(z;\tau),
    \end{align}
    where $l_j=n_{t_j}+\cdots+n_{t_{j+1}-1}+n_{q_j}+k_{q_j+1}+\cdots+k_{q_{j+1}-1}$, $t_{h+1}=r+1$ and $\zeta^{(z)}(k_1,\dots,k_r)$ denote the \emph{Hurwitz multiple zeta values} defined by
    $$\zeta^{(z)}(k_1,\dots,k_r)\ceq\sum_{0<n_1<\cdots<n_r}\frac{1}{(z+n_1)^{k_1}\cdots(z+n_r)^{k_r}}.$$
\end{prop}
\begin{proof}
    This follows from the same manipulations in Bachmann--Tasaka \cite{BT} that one obtains the Fourier expansion of the multiple Eisenstein series. The calculation is achieved by decomposing the defining series for $\tilde{\wp}_{k_1,\dots,k_r}$ into $2^r$ terms based on the ordering of the running indices, and applying partial fraction decomposition and Bouillot's reduction formula for the multitangent function (\cite{Bouillot}).
\end{proof}

By \Cref{thm:wp-coeff-two-index}, $\wp_{\{2\}^r}(z;\tau)$ can be written as $f_r(\tau)\wp_2(z;\tau)+g_r(\tau)$, where $f_r(\tau)$ and $g_r(\tau)$ are quasi-modular forms expressed via partition Eisenstein traces. The Weierstrass $\wp$-function is known as a meromorphic Jacobi form, and therefore, our $\wp_{\{2\}^r}(z;\tau)$ is an example of a meromorphic quasi-modular Jacobi form, defined in \cite{LLLW}. Recent studies (\cite{AGO,LLLW}) have shown that both meromorphic quasi-modular Jacobi forms and partition Eisenstein traces are related to a elliptic genus. Although it remains unclear whether our $\wp_{\{2\}^r}(z;\tau)$ has a direct connection to them, we summarize its basic properties, specifically, its Fourier expansion and modular transformation low, below.

\begin{prop}
For any $r>0$, we have
\begin{align}
    \wp_{\{2\}^r}(z;\tau)=\sum_{0\le{}i\le{}j\le{}r}c_{r,j}\;g_{\{2\}^i}(z;\tau)\left(g_{\{2\}^{j-i}}(-z;\tau)+\frac{(2\pi i)^2\xi}{(1-\xi)^2}g_{\{2\}^{j-i-1}}(-z;\tau)\right),
\end{align}
where $c_{r,j}=\frac{(-1)^{r-j}2^j}{(2\pi i)^{2r-2j}(2r)!}\sum_{\substack{m_1+\cdots+m_j=r\\m_1,\dots,m_j>0}}\binom{2r}{2m_1,\dots,2m_j}$.
\end{prop}

\begin{proof}
    From the results of Bouillot (\cite{Bouillot}, Theorem 1) and the explicit expressions for $\zeta(2,2,\dots,2)$, we have
    \begin{align}
        \wp_{\{2\}^r}(-z;\tau)&=\sum_{h=1}^r\sum_{\substack{l_1+\cdots+l_h=r\\l_1,\dots,l_h>0}}\sum_{m_1<\cdots<m_h}\Psi_{\{2\}^{l_1}}(z+m_1\tau)\cdots\Psi_{\{2\}^{l_h}}(z+m_h\tau)\\
        &=\sum_{h=1}^r\left(\sum_{m_1<\cdots<m_h}\prod_{i=1}^h\Psi_2(z+m_i)\right)\left(\sum_{\substack{l_1+\cdots+l_h=r\\l_1,\dots,l_h>0}}\frac{\pi^{2r-2h}\cdot2^{2r-h}}{(2l_1)!\cdots(2l_h)!}\right).
    \end{align}
    From the Lipschitz summation formula, we have the conclusion.
\end{proof}

\begin{lem}\label{lem:sigma modular}
    For $\begin{pmatrix}a&b\\c&d\end{pmatrix}\in\mathrm{SL}_2(\Z)$, we have $$\sigma\left(\frac{z}{c\tau+d};\frac{a\tau+b}{c\tau+d}\right)=\frac{1}{c\tau+d}\exp\left(\frac{\pi icz^2}{c\tau+d}\right)\sigma(z;\tau).$$
\end{lem}

\begin{proof}
    This follows from the definition of $\sigma$-function \cref{eq:def-sigma} and the modular transformations of Eisenstein series (see e.g. \cite{Zag}):
    \begin{align}
        G_{k}\left(\frac{a\tau+b}{c\tau+d}\right)=-\delta_{k,2}\cdot2\pi ic(c\tau+d)+(c\tau+d)^kG_k(\tau)
    \end{align}
    for $\begin{pmatrix}a&b\\c&d\end{pmatrix}\in\mathrm{SL}_2(\Z)$ and $k\geq2$, where $\delta_{k,2}$ is the Kronecker's delta.
\end{proof}

\begin{prop}
    For $r\ge0$ and $\begin{pmatrix}a&b\\c&d\end{pmatrix}\in\mathrm{SL}_2(\Z)$, we have
    $$(c\tau+d)^{-2r}\wp_{\{2\}^r}\left(\frac{z}{c\tau+d};\frac{a\tau+b}{c\tau+d}\right)=\sum_{j=0}^r\frac{1}{(r-j)!}\left(\frac{-2\pi ic}{c\tau+d}\right)^{r-j}\wp_{\{2\}^j}(z;\tau).$$
\end{prop}

\begin{proof}
    Let $$P_\tau(z,\alpha)\ceq\sum_{r\ge0}(-1)^r\wp_{\{2\}^r}(z;\tau)\alpha^{2r}=\frac{\sigma(z+\alpha)\sigma(z-\alpha)}{\sigma(z)^2}.$$
    From the \Cref{lem:sigma modular}, we have
    $$P_{\frac{a\tau+b}{c\tau+d}}\left(\frac{z}{c\tau+d},\frac{\alpha}{c\tau+d}\right)=\exp\left(\frac{2\pi ic\alpha^2}{c\tau+d}\right)P_\tau(z,\alpha).$$
    Comparing the coefficients of $\alpha^{2r}$ of both sides, we have the conclusion.
\end{proof}

\section*{Acknowledgements}

The first named author expresses his gratitude to his academic advisor Professor Yasuo Ohno. The second named author expresses his gratitude to his academic advisor Professor Masanobu Kaneko. The authors would like to thank Professor Toshiki Matsusaka for his valuable comments on this paper.


\begin{thebibliography}{00}



\bibitem{AGOS} T. Amdeberhan, M. Griffin, K. Ono, and A. Singh, \textit{Traces of partition Eisenstein series}, Forum Mathematicum, \textbf{37} (2025), No. 6, 1835--1847.
\bibitem{AGO} T. Amdeberhan, M. Griffin and K. Ono, \textit{Some topological genera and Jacobi forms}, Proc. Natl. Acad. Sci. USA {\bf 122} (2025), no.~32, Paper No. e2502678122, 11 pp.
\bibitem{AOS} T. Amdeberhan, K. Ono, and A. Singh, \textit{Derivatives of Theta Functions as Traces of Partition Eisenstein Series}, Nagoya Mathematical Journal. Published online 2025:1--12.
\bibitem{BIM} H. Bachmann, J.-W. van Ittersum, \textit{Formal multiple Eisenstein series and their derivations}, Adv. Math. {\bf 487} (2026), Paper No. 110739, 52 pp.
\bibitem{BKM} H. Bachmann, H. Kanno, T. Maesaka, Relations and Derivatives of Multiple Eisenstein Series, arXiv:2602.08176.
\bibitem{BK} H. Bachmann, U. K{\"u}hn, A dimension conjecture for q-analogues of multiple zeta values, Periods in Quantum Field Theory and Arithmetic, Springer Proceedings in Mathematics \& Statistics \textbf{314} (2020),  237--258.
\bibitem{BT} H. Bachmann and K. Tasaka, \textit{The double shuffle relations for multiple Eisenstein series}, Nagoya Math. J., \textbf{230} (2018), 180--212.
\bibitem{Bouillot} O. Bouillot, \textit{The algebra of multitangent functions}, J. Algebra \textbf{410} (2014), 148--238.
\bibitem{Bringmann} K. Bringmann, B.V. Pandey, \textit{Traces of partition Eisenstein series and almost holomorphic modular forms}. Res. number theory \textbf{11}, 49 (2025).
\bibitem{GKZ} H. Gangl, M. Kaneko, and D. Zagier, \textit{Double zeta values and modular forms}, Automorphic forms and Zeta functions”, Proceedings of the conference in memory of Tsuneo Arakawa, World Scientific, (2006), 71--106.
\bibitem{Hirose} M. Hirose, \textit{Multitangent functions and symmetric multiple zeta values}, arXiv:2402.13902.
\bibitem{Hirose2} M. Hirose, \textit{An explicit parity theorem for multiple zeta values via multitangent functions}, Ramanujan J., \textbf{67}, 87 (2025).
\bibitem{Hoff} M.E. Hoffman, Multiple Harmonic Series, Pacific Journal of Mathematics, Vol. \textbf{152}, No. 2 (1992).
\bibitem{HI} M.E. Hoffman and K. Ihara, \textit{Quasi-shuffle products revisited}, J. Algebra, \textbf{481} (2017), 293--326.
\bibitem{Kaneko04} M. Kaneko, \textit{Double zeta values and modular forms}, proceedings of the Japan--Korea joint seminar on Number Theory (Kuju, Japan) (H. K. Kim and Y. Taguchi, eds.), October 2004.
\bibitem{KT} M. Kaneko and K. Tasaka, \textit{Double zeta values, double Eisenstein series, and modular forms of level 2}, Math. Ann. (2013) \textbf{357}:1091--1118.
\bibitem{Kanno} H. Kanno, \textit{Shuffle regularization for multiple Eisenstein series of level N}, Ramanujan J. \textbf{67} (2025), no. 4, 95.
\bibitem{Kina} K. Kina, \textit{Double Eisenstein series and modular forms of level 4}, Ramanujan J. (2024) \textbf{65}:1651--1696.
\bibitem{Lang} S. Lang, \textit{Elliptic Functions}, second edition, Graduate Texts in Mathematics, 112, Springer, New York, 1987.
\bibitem{LLLW} S.-J. Lee, W. Lerche, G. Lockhart and T. Weigand,\textit{Quasi-Jacobi forms, elliptic genera and strings in four dimensions}, Journal of High Energy Physics, vol. \textbf{2021}, no. 162 (2021)
\bibitem{Mat} T. Matsusaka, \textit{Applications of Fa\`{a} di Bruno’s formula to partition traces}, Res. number theory \textbf{11}, 69 (2025).
\bibitem{Oko} A. Okounkov, \textit{Hilbert schemes and multiple q-zeta values}, Funct Anal Its Appl \textbf{48}, 138--144 (2014).
\bibitem{Stan} R.P. Stanley, \textit{Enumerative combinatorics. Vol. 2}, Cambridge Studies in Math. \textbf{62} Cambridge Univ. Press, 1999.
\bibitem{Tasaka} K. Tasaka, \textit{Hecke eigenform and double Eisenstein series}, Proceedings of the American Mathematical Society \textbf{148} (2020), No. 1, 53--58.
\bibitem{YZ} H. Yuan and J. Zhao, \textit{Double shuffle relations of double zeta values and the double Eisenstein series at level N}, J. Lond. Math. Soc. (2) \textbf{92} (2015), No. 3, 520--546.
\bibitem{Zag} D. Zagier, Elliptic modular forms and their applications, in {\textit The 1-2-3 of modular forms}, 1--103, Universitext, Springer, Berlin.

\end{thebibliography}



\end{document}